\documentclass[9pt]{article}%
\usepackage{bbm}
\usepackage{epsfig}
\usepackage{graphics,graphicx,amssymb,amsmath,verbatim}
\usepackage{amssymb}
\usepackage{mathrsfs}
\usepackage{amsfonts}
\usepackage{amsmath}
\usepackage{graphicx}
\usepackage{easybmat}
\usepackage{color}
\usepackage{mathrsfs}%
\providecommand{\U}[1]{\protect \rule{.1in}{.1in}}
\newtheorem{theorem}{Theorem}

\newtheorem{corollary}{Corollary}

\newtheorem{lemma}{Lemma}

\newtheorem{proposition}{Proposition}

\newenvironment{proof}[1][Proof]{\noindent \textbf{#1.} }{\  \rule{0.5em}{0.5em}}
\definecolor{blue}{rgb}{0,0,1}
\allowdisplaybreaks[4]
\begin{document}

\title{On Strong Stability and Robust Strong Stability of Linear Difference Equations
with Two Delays}
\author{Bin Zhou\thanks{The author is with the Center for Control Theory and Guidance
Technology, Harbin Institute of Technology, Harbin, 150001, China. Email:
\texttt{binzhoulee@163.com, binzhou@hit.edu.cn}.}}
\date{}
\maketitle

\begin{abstract}
This paper provides a necessary and sufficient condition for guaranteeing
exponential stability of the linear difference equation $x(t)=Ax(t-a)+Bx(t-b)$
where $a>0,b>0$ are constants and $A,B$ are $n\times n$ square matrices, in
terms of a linear matrix inequality (LMI) of size $\left(  k+1\right)
n\times \left(  k+1\right)  n$ where $k\geq1$ is some integer. Different from
an existing condition where the coefficients $\left(  A,B\right)  $ appear as
highly nonlinear functions, the proposed LMI condition involves matrices that
are linear functions of $\left(  A,B\right)  .$ Such a property is further
used to deal with the robust stability problem in case of norm bounded
uncertainty and polytopic uncertainty, and the state feedback stabilization
problem. Solutions to these two problems are expressed by LMIs. A time domain
interpretation of the proposed LMI condition in terms of Lyapunov-Krasovskii
functional is given, which helps to reveal the relationships among the
existing methods. Numerical example demonstrates the effectiveness of the
proposed method.

\vspace{0.3cm}

\textbf{Keywords:} Linear difference equations; Exponential stability;
Necessary and sufficient conditions; Linear matrix inequality.

\end{abstract}

\section{Introduction and Literature Review}

Throughout this paper, we use $A\otimes B$ to denote the Kronecker product of
matrices $A$ and $B.$ For a matrix $A,$ the symbols $\left \vert A\right \vert
,\left \Vert A\right \Vert ,A^{\mathrm{T}},A^{\mathrm{H}}$, and $\rho \left(
A\right)  $ denote respectively its determinant, norm, transpose, conjugate
transpose, and spectral radius. For a square matrix $P$, $P>0$ denotes that it
is positive definite.

The linear (continuous-time) difference equation%
\begin{equation}
x(t)=\sum \limits_{i=1}^{N}A_{i}x(t-r_{i}), \label{sys0}%
\end{equation}
where $r_{i}>0$ are constants and $A_{i}$ are square matrices, is frequently
encountered in neutral-type time-delay systems \cite{gu12am,mvznh09siam} and
coupled differential-functional equations \cite{gu10auto,lg10auto}. The
stability of system (\ref{sys0}) is usually the necessary condition for
ensuring the asymptotic stability of the above two types of time-delay
systems, and thus has attracted considerable attentions in the literature
\cite{carvalho96laa,ddmb16tac,fridman02jmaa,gu10auto,hale71book,pepe14auto,rmd18tac}%
.

It is known that (\ref{sys0}) is stable if and only if its spectral abscissa
is less than zero \cite{hale71book}. However, the spectral abscissa of
(\ref{sys0}) is not continuous in delays and the stability might be destroyed
by arbitrarily small changes in the delay \cite{ah80jmaa,hale71book}.
Therefore, the concept of strong stability was introduced by \cite{hale71book}
to handle this hypersensitivity of the stability with respect to delays, which
has been generalized in \cite{mvznh09siam}. To go further, we introduce the
following result from Theorem 6.1 (Chapter 9, p. 286) in \cite{hale71book}.

\begin{lemma}
\label{lm0}System (\ref{sys0}) is strongly stable if and only if%
\begin{equation}
\max_{\theta_{i}\in \left[  0,2\pi \right]  ,i=1,2,\ldots,N}\rho \left(
\sum \limits_{i=1}^{N}A_{i}\mathrm{e}^{\mathrm{j}\theta_{i}}\right)
<1.\label{eqr11}%
\end{equation}

\end{lemma}

The strong stability concept is important since in practical applications the
delays are generally subject to small errors \cite{gu10auto}. The test of
strong stability is however rather complex \cite{hv12auto}. Indeed, condition
(\ref{eqr11}) is not tractable in general since the spectral radius should be
tested for all $\theta_{i}\in \left[  0,2\pi \right]  ,i=1,2,\ldots,N.$ Strong
stability of (\ref{sys0}) was tested via deciding positive definiteness of a
multivariate trigonometric polynomial matrix, which is then solved as a
converging hierarchy of LMIs \cite{hv12auto}. The condition in \cite{hv12auto}
needs to compute the characteristic equation of (\ref{sys0}), which is not
explicitly expressed as functions of the coefficients, and thus seems
difficult to be used for robust stability analysis. For a single delay, strong
stability can be checked by computing the generalized eigenvalues of a pair of
matrices \cite{louisell01tac,louisell15siam} as well as the matrix pencil
based approach \cite{niculescu01book}. The method of cluster treatment of
characteristic roots was used in \cite{ovs08siam} to derive the stability maps
of (\ref{sys0}) with three delays. For more related work, see
\cite{gu10auto,hale71book,hv12auto,ovs08siam} and the references therein.

In this note, we restrict ourself to a special case of (\ref{sys0}) where
$N=2$, for which we rewrite (\ref{sys0}) as%
\begin{equation}
x(t)=Ax(t-a)+Bx(t-b), \label{sys}%
\end{equation}
where $a,b$ are positive constants, and $A,B$ are $n\times n$ square matrices.
Regarding the existence of a solution, the continuity/discontinuity of the
solution, and definitions for stability of the solution, readers are suggested
to refer \cite{carvalho96laa} and \cite{hale71book} for details. Notice that,
by Lemma \ref{lm0}, system (\ref{sys}) is strongly stable if and only if%
\begin{equation}
\rho \left(  \Delta_{\theta}\right)  <1,\; \theta \in \left[  0,2\pi \right]
,\text{ }\Delta_{\theta}=A+B\mathrm{e}^{-\mathrm{j}\theta}. \label{eqr}%
\end{equation}

It came to our attention that condition (\ref{eqr}) happens to be equivalent
to the stability of the 2-D linear system described by the
Fornasini-Marchesini second model%
\begin{equation}
x\left(  i+1,j+1\right)  =Ax\left(  i,j+1\right)  +Bx\left(  i+1,j\right)  ,
\label{eq80}%
\end{equation}
which has been well studied in the literature \cite{eih06tac,fcn06auto}. For
stability analysis of (\ref{eq80}), a necessary and sufficient condition
expressed by an LMI of size $3n^{2}\times3n^{2}$ was established in
\cite{eih06tac}.

\begin{lemma}
\label{th1} The system (\ref{sys})/(\ref{eq80}) is strongly/exponentially
stable if and only if
\begin{equation}
\rho \left(  A+B\right)  <1, \label{eq0}%
\end{equation}
and there exist two symmetric matrices $P_{1}\in \mathbf{R}^{n^{2}\times n^{2}%
},P_{2}\in \mathbf{R}^{n^{2}\times n^{2}}$ and a matrix $P_{3}\in
\mathbf{R}^{n^{2}\times n^{2}}$ such that%
\begin{equation}
\left[
\begin{array}
[c]{ccc}%
-P_{1} & 0 & -P_{3}\\
0 & -P_{2} & P_{3}^{\mathrm{T}}\\
-P_{3}^{\mathrm{T}} & P_{3} & P_{1}+P_{2}%
\end{array}
\right]  <E^{\mathrm{T}}E, \label{eq1}%
\end{equation}
where $E=[B^{\mathrm{T}}\otimes A,A^{\mathrm{T}}\otimes B,A^{\mathrm{T}%
}\otimes A+B^{\mathrm{T}}\otimes B-I_{n}\otimes I_{n}].$
\end{lemma}

This result is almost the same as Theorem 1 in \cite{eih06tac}, where $E$ is
replaced by $E_{\ast}=[B\otimes A,A\otimes B,A\otimes A+B\otimes
B-I_{n}\otimes I_{n}].$ The proof given in \cite{eih06tac} is based on the
Guardian map and the positive real lemma. Motivated by \cite{ddp18auto}, we
provide in Appendix a simple proof based on the well-known
Yakubovich-Kalman-Popov (YKP) lemma. Another necessary and sufficient
conditions, which involve the generalized eigenvalues of two matrices with
size $2n^{2}\times2n^{2},$ were obtained in \cite{fcn06auto}, which were also
established initially for testing stability of the 2-D linear system
(\ref{eq80}).

Bliman established in \cite{bliman02mssp} another LMI based necessary and
sufficient conditions for testing stability of (\ref{eq80}). To introduce this
result, for any $k\in \mathbf{N}^{+},$ we define
\begin{align}
\overline{A}_{k}  &  =\left[
\begin{array}
[c]{ccccc}%
0 & B & AB & \cdots & A^{k-2}B\\
& 0 & B & \cdots & A^{k-3}B\\
&  & \ddots & \ddots & \vdots \\
&  &  & 0 & B\\
&  &  &  & 0
\end{array}
\right]  \in \mathbf{R}^{kn\times kn},\overline{B}_{k}=\left[
\begin{array}
[c]{c}%
A^{k-1}\\
A^{k-2}\\
\vdots \\
A\\
I_{n}%
\end{array}
\right]  \in \mathbf{R}^{kn\times n},\label{eqab1}\\
\overline{\mathscr{A}}_{k}  &  =\left[
\begin{array}
[c]{ccccc}%
B & AB & A^{2}B & \cdots & A^{k-1}B\\
& B & AB & \cdots & A^{k-2}B\\
&  & \ddots & \ddots & \vdots \\
&  &  & B & AB\\
&  &  &  & B
\end{array}
\right]  \in \mathbf{R}^{kn\times kn},\overline{\mathscr{B}}_{k}=\left[
\begin{array}
[c]{c}%
A^{k}\\
A^{k-1}\\
\vdots \\
A^{2}\\
A
\end{array}
\right]  \in \mathbf{R}^{kn\times n}. \label{eqab2}%
\end{align}
For two symmetric matrices $\overline{P},\overline{Q}\in \mathbf{R}^{kn\times
kn},$ we define a linear function $\overline{\mathit{\Omega}}_{k}\left(
\overline{P},\overline{Q}\right)  \in \mathbf{R}^{\left(  k+1\right)
n\times \left(  k+1\right)  n}$ as%
\begin{equation}
\overline{\mathit{\Omega}}_{k}\left(  \overline{P},\overline{Q}\right)
=\left[
\begin{array}
[c]{cc}%
\overline{A}_{k}^{\mathrm{T}}\overline{P}\overline{A}_{k}-\overline
{P}+\overline{\mathscr{A}}_{k}^{\mathrm{T}}\overline{Q}\overline
{\mathscr{A}}_{k}-\overline{A}_{k}^{\mathrm{T}}\overline{Q}\overline{A}_{k} &
\overline{A}_{k}^{\mathrm{T}}\overline{P}\overline{B}_{k}+\overline
{\mathscr{A}}_{k}^{\mathrm{T}}\overline{Q}\overline{\mathscr{B}}_{k}%
-\overline{A}_{k}^{\mathrm{T}}\overline{Q}\overline{B}_{k}\\
\overline{B}_{k}^{\mathrm{T}}\overline{P}\overline{A}_{k}+\overline
{\mathscr{B}}_{k}^{\mathrm{T}}\overline{Q}\overline{\mathscr{A}}_{k}%
-\overline{B}_{k}^{\mathrm{T}}\overline{Q}\overline{A}_{k} & \overline{B}%
_{k}^{\mathrm{T}}\overline{P}\overline{B}_{k}+\overline{\mathscr{B}}%
_{k}^{\mathrm{T}}\overline{Q}\overline{\mathscr{B}}_{k}-\overline{B}%
_{k}^{\mathrm{T}}\overline{Q}\overline{B}_{k}%
\end{array}
\right]  . \label{eqomegak}%
\end{equation}

\begin{lemma}
\label{th2}\cite{bliman02mssp} If there exist positive definite matrices
$\overline{P}_{k},\overline{Q}_{k}\in \mathbf{R}^{kn\times kn}$ such that%
\begin{equation}
\overline{\mathit{\Omega}}_{k}\left(  \overline{P}_{k},\overline{Q}%
_{k}\right)  <0, \label{eqz}%
\end{equation}
then system (\ref{sys})/(\ref{eq80}) is stable. Moreover, if (\ref{sys}%
)/(\ref{eq80}) is stable, there exists an integer $k^{\ast}\geq1$, such that
(\ref{eqz}) is solvable with $\overline{P}_{k}>0,\overline{Q}_{k}>0,\forall
k\geq k^{\ast}.$
\end{lemma}

Notice that Lemma \ref{th2} is slightly different from the original one in
\cite{bliman02mssp} where the result is built for a general 2-D linear system,
and is expressed in a recursive form. Even for $k=2,$ the LMI in Lemma
\ref{th2} is nonlinear in $A$ and $B$, and thus can not be used for robust
stability analysis.

In this note, motivated by \cite{bliman02mssp}, we will establish a new
necessary and sufficient condition for testing strong stability of system
(\ref{sys}). Different from Lemmas \ref{th1} and \ref{th2}, the proposed LMI
condition involves matrices that are linear functions of $\left(  A,B\right)
.$ With the help of this property, the robust stability problem in case of
norm bounded uncertainty is investigated, and the results are also expressed
by LMIs (see Section \ref{sec2}). We also give time-domain interpretations of
the proposed LMI condition and the Bliman condition, which help to reveal the
relationships among them and the other existing methods such as those in
\cite{carvalho96laa} and \cite{ddmb16tac} (see Section \ref{sec3}).

\section{\label{sec2}The Necessary and Sufficient Conditions}

For any $k\in \mathbf{N}^{+},$ we denote%
\begin{equation}
A_{k}=\left[
\begin{array}
[c]{cc}%
0 & I_{\left(  k-1\right)  n}\\
0 & 0
\end{array}
\right]  \in \mathbf{R}^{kn\times kn},\;B_{k}=\left[
\begin{array}
[c]{c}%
0\\
I_{n}%
\end{array}
\right]  \in \mathbf{R}^{kn\times n}, \label{eqab3}%
\end{equation}
and
\begin{equation}
L_{k}=\left[
\begin{array}
[c]{cc}%
I_{kn} & 0_{kn\times n}%
\end{array}
\right]  \in \mathbf{R}^{kn\times \left(  k+1\right)  n}, \label{eqlk}%
\end{equation}
which are independent of $\left(  A,B\right)  ,$ and%
\begin{equation}
\mathscr{A}_{k}=\left[
\begin{array}
[c]{ccccc}%
B & A & 0 & \cdots & 0\\
& B & A & \ddots & \vdots \\
&  & \ddots & \ddots & 0\\
&  &  & B & A\\
&  &  &  & B
\end{array}
\right]  \in \mathbf{R}^{kn\times kn},\; \mathscr{B}_{k}=\left[
\begin{array}
[c]{c}%
0\\
0\\
\vdots \\
0\\
A
\end{array}
\right]  \in \mathbf{R}^{kn\times n}, \label{eqab4}%
\end{equation}
which are linear matrix functions of $\left(  A,B\right)  .$ For two symmetric
matrices $P,Q\in \mathbf{R}^{kn\times kn},$ we define%
\begin{align}
\mathit{\Omega}_{k1}\left(  P,Q\right)  =  &  \left[
\begin{array}
[c]{cc}%
A_{k} & B_{k}%
\end{array}
\right]  ^{\mathrm{T}}\left(  P-Q\right)  \left[
\begin{array}
[c]{cc}%
A_{k} & B_{k}%
\end{array}
\right]  -L_{k}^{\mathrm{T}}PL_{k},\nonumber \\
\mathit{\Omega}_{k}\left(  P,Q\right)  =  &  \mathit{\Omega}_{k1}\left(
P,Q\right)  +\left[
\begin{array}
[c]{cc}%
\mathscr{A}_{k} & \mathscr{B}_{k}%
\end{array}
\right]  ^{\mathrm{T}}Q\left[
\begin{array}
[c]{cc}%
\mathscr{A}_{k} & \mathscr{B}_{k}%
\end{array}
\right]  , \label{eq77b}%
\end{align}
which are linear functions of $P,Q$ and, moreover, $\mathit{\Omega}%
_{k1}\left(  P,Q\right)  $ is independent of $\left(  A,B\right)  .$

\begin{theorem}
\label{th4}If there exist positive definite matrices $P_{k},Q_{k}\in
\mathbf{R}^{kn\times kn}$ such that%
\begin{equation}
\mathit{\Omega}_{k}\left(  P_{k},Q_{k}\right)  <0, \label{eq55}%
\end{equation}
then system (\ref{sys}) is strongly stable. Moreover, if system (\ref{sys}) is
strongly stable, there exists an integer $k^{\ast}\geq1$, such that
(\ref{eq55}) is solvable with $P_{k}>0,Q_{k}>0,\forall k\geq k^{\ast}.$
\end{theorem}

\begin{proof}
Let
\begin{equation}
z_{k}=\left[
\begin{array}
[c]{c}%
z_{k,k}\\
\vdots \\
z_{k,2}\\
z_{k,1}%
\end{array}
\right]  =\left(  \mathrm{e}^{\mathrm{j}\theta}I_{kn}-A_{k}\right)  ^{-1}%
B_{k}, \label{eqzk}%
\end{equation}
which is equivalent to%
\[
\left[
\begin{array}
[c]{ccccc}%
\mathrm{e}^{\mathrm{j}\theta}I_{n} & -I_{n} & 0 & \cdots & 0\\
& \mathrm{e}^{\mathrm{j}\theta}I_{n} & \ddots & \ddots & \vdots \\
&  & \ddots & -I_{n} & 0\\
&  &  & \mathrm{e}^{\mathrm{j}\theta}I_{n} & -I_{n}\\
&  &  &  & \mathrm{e}^{\mathrm{j}\theta}I_{n}%
\end{array}
\right]  \left[
\begin{array}
[c]{c}%
z_{k,k}\\
z_{k,k-1}\\
\vdots \\
z_{k,2}\\
z_{k,1}%
\end{array}
\right]  =\left[
\begin{array}
[c]{c}%
0\\
0\\
\vdots \\
0\\
I_{n}%
\end{array}
\right]  .
\]
Solving this equation recursively from the bottom to the up gives
\begin{equation}
z_{k}=\left[
\begin{array}
[c]{c}%
\mathrm{e}^{-k\mathrm{j}\theta}I_{n}\\
\vdots \\
\mathrm{e}^{-2\mathrm{j}\theta}I_{n}\\
\mathrm{e}^{-\mathrm{j}\theta}I_{n}%
\end{array}
\right]  . \label{eqzk2}%
\end{equation}
With this we get from (\ref{eqab3}), (\ref{eqlk}) and (\ref{eqab4}) that%
\begin{align}
\left[
\begin{array}
[c]{cc}%
A_{k} & B_{k}%
\end{array}
\right]  \left[
\begin{array}
[c]{c}%
\left(  \mathrm{e}^{\mathrm{j}\theta}I_{kn}-A_{k}\right)  ^{-1}B_{k}\\
I_{n}%
\end{array}
\right]   &  =\left[
\begin{array}
[c]{cc}%
A_{k} & B_{k}%
\end{array}
\right]  \left[
\begin{array}
[c]{c}%
z_{k}\\
I_{n}%
\end{array}
\right]  =\left[
\begin{array}
[c]{c}%
\mathrm{e}^{-\mathrm{j}\left(  k-1\right)  \theta}I_{n}\\
\vdots \\
\mathrm{e}^{-\mathrm{j}\theta}I_{n}\\
I_{n}%
\end{array}
\right]  =\mathrm{e}^{\mathrm{j}\theta}z_{k},\label{eqabx1}\\
\left[
\begin{array}
[c]{cc}%
\mathscr{A}_{k} & \mathscr{B}_{k}%
\end{array}
\right]  \left[
\begin{array}
[c]{c}%
\left(  \mathrm{e}^{\mathrm{j}\theta}I_{kn}-A_{k}\right)  ^{-1}B_{k}\\
I_{n}%
\end{array}
\right]   &  =\left[
\begin{array}
[c]{cc}%
\mathscr{A}_{k} & \mathscr{B}_{k}%
\end{array}
\right]  \left[
\begin{array}
[c]{c}%
z_{k}\\
I_{n}%
\end{array}
\right]  =\left[
\begin{array}
[c]{c}%
\mathrm{e}^{-\mathrm{j}\left(  k-1\right)  \theta}\Delta_{\theta}\\
\vdots \\
\mathrm{e}^{-\mathrm{j}\theta}\Delta_{\theta}\\
\Delta_{\theta}%
\end{array}
\right]  =\mathrm{e}^{\mathrm{j}\theta}z_{k}\Delta_{\theta}, \label{eqabx2}%
\end{align}
and%
\begin{equation}
L_{k}\left[
\begin{array}
[c]{c}%
\left(  \mathrm{e}^{\mathrm{j}\theta}I_{kn}-A_{k}\right)  ^{-1}B_{k}\\
I_{n}%
\end{array}
\right]  =\left(  \mathrm{e}^{\mathrm{j}\theta}I_{kn}-A_{k}\right)  ^{-1}%
B_{k}=z_{k}. \label{eqabx3}%
\end{equation}
Therefore, we can obtain%
\begin{align}
&  \left[
\begin{array}
[c]{c}%
z_{k}\\
I_{n}%
\end{array}
\right]  ^{\mathrm{H}}\mathit{\Omega}_{k}\left(  P_{k},Q_{k}\right)  \left[
\begin{array}
[c]{c}%
z_{k}\\
I_{n}%
\end{array}
\right] \nonumber \\
=  &  \left[
\begin{array}
[c]{c}%
z_{k}\\
I_{n}%
\end{array}
\right]  ^{\mathrm{H}}\left(  \left[
\begin{array}
[c]{cc}%
A_{k} & B_{k}%
\end{array}
\right]  ^{\mathrm{T}}P_{k}\left[
\begin{array}
[c]{cc}%
A_{k} & B_{k}%
\end{array}
\right]  -L_{k}^{\mathrm{T}}P_{k}L_{k}\right)  \left[
\begin{array}
[c]{c}%
z_{k}\\
I_{n}%
\end{array}
\right] \nonumber \\
&  +\left[
\begin{array}
[c]{c}%
z_{k}\\
I_{n}%
\end{array}
\right]  ^{\mathrm{H}}\left[
\begin{array}
[c]{cc}%
\mathscr{A}_{k} & \mathscr{B}_{k}%
\end{array}
\right]  ^{\mathrm{T}}Q_{k}\left[
\begin{array}
[c]{cc}%
\mathscr{A}_{k} & \mathscr{B}_{k}%
\end{array}
\right]  \left[
\begin{array}
[c]{c}%
z_{k}\\
I_{n}%
\end{array}
\right] \nonumber \\
&  -\left[
\begin{array}
[c]{c}%
z_{k}\\
I_{n}%
\end{array}
\right]  ^{\mathrm{H}}\left[
\begin{array}
[c]{cc}%
A_{k} & B_{k}%
\end{array}
\right]  ^{\mathrm{T}}Q_{k}\left[
\begin{array}
[c]{cc}%
A_{k} & B_{k}%
\end{array}
\right]  \left[
\begin{array}
[c]{c}%
z_{k}\\
I_{n}%
\end{array}
\right] \nonumber \\
=  &  \left(  \mathrm{e}^{\mathrm{j}\theta}z_{k}\Delta_{\theta}\right)
^{\mathrm{H}}Q_{k}\mathrm{e}^{\mathrm{j}\theta}z_{k}\Delta_{\theta}-\left(
\mathrm{e}^{\mathrm{j}\theta}z_{k}\right)  ^{\mathrm{H}}Q_{k}\mathrm{e}%
^{\mathrm{j}\theta}z_{k}+\left(  \mathrm{e}^{\mathrm{j}\theta}z_{k}\right)
^{\mathrm{H}}P_{k}\mathrm{e}^{\mathrm{j}\theta}z_{k}-z_{k}^{\mathrm{H}}%
P_{k}z_{k}\nonumber \\
=  &  \Delta_{\theta}^{\mathrm{H}}z_{k}^{\mathrm{H}}Q_{k}z_{k}\Delta_{\theta
}-z_{k}^{\mathrm{H}}Q_{k}z_{k}\nonumber \\
<  &  0, \label{eq78}%
\end{align}
which implies (\ref{eqr}) since $z_{k}^{\mathrm{H}}Q_{k}z_{k}>0.$

We next prove the converse. By Lemma \ref{lm9} in Appendix A2, we know that
there exists a $k^{\ast}\geq1$ such that%
\begin{equation}
\left(  \Delta_{\theta}^{k}\right)  ^{\mathrm{H}}\Delta_{\theta}^{k}<I_{n},\;
\forall \theta \in \left[  0,2\pi \right]  ,\;k\geq k^{\ast}. \label{eq68}%
\end{equation}
Denote $Q_{k}^{\ast}=W_{k}^{\mathrm{T}}W_{k}$ where (see the notation in
Appendix A2)%
\begin{equation}
W_{k}=\left[
\begin{array}
[c]{ccccc}%
B^{\left[  k-1\right]  } & B^{\left[  k-2\right]  }A^{\left[  1\right]  } &
\cdots & B^{\left[  1\right]  }A^{\left[  k-2\right]  } & A^{\left[
k-1\right]  }\\
& B^{\left[  k-2\right]  } & B^{\left[  k-3\right]  }A^{\left[  1\right]  } &
\ddots & A^{\left[  k-2\right]  }\\
&  & \ddots & \ddots & \vdots \\
&  &  & B^{\left[  1\right]  } & A^{\left[  1\right]  }\\
&  &  &  & I_{n}%
\end{array}
\right]  . \label{eqwk}%
\end{equation}
It follows that $Q_{k}^{\ast}\geq0$ and, moreover, $Q_{k}^{\ast}>0$ if $B$ is
nonsingular. For any integer $i\geq1,$ by the binomial expansion theorem, we
have%
\begin{align*}
\left(  A+B\mathrm{e}^{-\mathrm{j}\theta}\right)  ^{i}  &  =A^{\left[
i\right]  }+B^{\left[  1\right]  }A^{\left[  i-1\right]  }\mathrm{e}%
^{-\mathrm{j}\theta}+B^{\left[  2\right]  }A^{\left[  i-2\right]  }%
\mathrm{e}^{-2\mathrm{j}\theta}+\cdots+B^{\left[  i-1\right]  }A^{\left[
1\right]  }\mathrm{e}^{-\left(  i-1\right)  \mathrm{j}\theta}+B^{\left[
i\right]  }\mathrm{e}^{-i\mathrm{j}\theta}\\
&  =\left[
\begin{array}
[c]{ccccc}%
B^{\left[  i\right]  } & B^{\left[  i-1\right]  }A^{\left[  1\right]  } &
\cdots & B^{\left[  1\right]  }A^{\left[  i-1\right]  } & A^{\left[  i\right]
}%
\end{array}
\right]  \left[
\begin{array}
[c]{c}%
\mathrm{e}^{-i\mathrm{j}\theta}I_{n}\\
\mathrm{e}^{-\left(  i-1\right)  \mathrm{j}\theta}I_{n}\\
\vdots \\
\mathrm{e}^{-\mathrm{j}\theta}I_{n}\\
I_{n}%
\end{array}
\right]  .
\end{align*}
It follows that%
\[
W_{k}\mathrm{e}^{\mathrm{j}\theta}z_{k}=W_{k}\left[
\begin{array}
[c]{c}%
\mathrm{e}^{-\left(  k-1\right)  \mathrm{j}\theta}I_{n}\\
\mathrm{e}^{-\left(  k-2\right)  \mathrm{j}\theta}I_{n}\\
\vdots \\
\mathrm{e}^{-\mathrm{j}\theta}I_{n}\\
I_{n}%
\end{array}
\right]  =\left[
\begin{array}
[c]{c}%
\Delta_{\theta}^{k-1}\\
\Delta_{\theta}^{k-2}\\
\vdots \\
\Delta_{\theta}\\
I_{n}%
\end{array}
\right]  .
\]
Let%
\[
\mathit{\Theta}_{k}\left(  Q\right)  =\left[
\begin{array}
[c]{cc}%
\mathscr{A}_{k} & \mathscr{B}_{k}%
\end{array}
\right]  ^{\mathrm{T}}Q\left[
\begin{array}
[c]{cc}%
\mathscr{A}_{k} & \mathscr{B}_{k}%
\end{array}
\right]  -\left[
\begin{array}
[c]{cc}%
A_{k} & B_{k}%
\end{array}
\right]  ^{\mathrm{T}}Q\left[
\begin{array}
[c]{cc}%
A_{k} & B_{k}%
\end{array}
\right]  .
\]
We then have from (\ref{eq68}) and equations (\ref{eqabx1}) and (\ref{eqabx2})
that%
\begin{align*}
&  \left[
\begin{array}
[c]{c}%
z_{k}\\
I_{n}%
\end{array}
\right]  ^{\mathrm{H}}\mathit{\Theta}_{k}\left(  Q_{k}^{\ast}\right)  \left[
\begin{array}
[c]{c}%
z_{k}\\
I_{n}%
\end{array}
\right] \\
=  &  \left(  \mathrm{e}^{\mathrm{j}\theta}z_{k}\Delta_{\theta}\right)
^{\mathrm{H}}Q_{k}^{\ast}\mathrm{e}^{\mathrm{j}\theta}z_{k}\Delta_{\theta
}-\left(  \mathrm{e}^{\mathrm{j}\theta}z_{k}\right)  ^{\mathrm{H}}Q_{k}^{\ast
}\mathrm{e}^{\mathrm{j}\theta}z_{k}\\
=  &  \Delta_{\theta}^{\mathrm{H}}\left(  \left(  W_{k}\mathrm{e}%
^{\mathrm{j}\theta}z_{k}\right)  ^{\mathrm{H}}W_{k}\mathrm{e}^{\mathrm{j}%
\theta}z_{k}\right)  \Delta_{\theta}-\left(  W_{k}\mathrm{e}^{\mathrm{j}%
\theta}z_{k}\right)  ^{\mathrm{H}}W_{k}\mathrm{e}^{\mathrm{j}\theta}z_{k}\\
=  &  \Delta_{\theta}^{\mathrm{H}}\left[
\begin{array}
[c]{c}%
\Delta_{\theta}^{k-1}\\
\vdots \\
\Delta_{\theta}\\
I_{n}%
\end{array}
\right]  ^{\mathrm{H}}\left[
\begin{array}
[c]{c}%
\Delta_{\theta}^{k-1}\\
\vdots \\
\Delta_{\theta}\\
I_{n}%
\end{array}
\right]  \Delta_{\theta}-\left[
\begin{array}
[c]{c}%
\Delta_{\theta}^{k-1}\\
\vdots \\
\Delta_{\theta}\\
I_{n}%
\end{array}
\right]  ^{\mathrm{H}}\left[
\begin{array}
[c]{c}%
\Delta_{\theta}^{k-1}\\
\vdots \\
\Delta_{\theta}\\
I_{n}%
\end{array}
\right] \\
=  &  \left(  \Delta_{\theta}^{k}\right)  ^{\mathrm{H}}\Delta_{\theta}%
^{k}-I_{n}\\
<  &  0.
\end{align*}
As $A_{k}$ is Schur stable, by the YKP lemma in Appendix A2, the above
inequality holds true if and only if there exists a symmetric matrix
$P_{k}^{\ast}\in \mathbf{R}^{kn\times kn}$ such that%
\begin{align}
0>  &  \left[
\begin{array}
[c]{cc}%
A_{k}^{\mathrm{T}}P_{k}^{\ast}A_{k}-P_{k}^{\ast} & A_{k}^{\mathrm{T}}%
P_{k}^{\ast}B_{k}\\
B_{k}^{\mathrm{T}}P_{k}^{\ast}A_{k} & B_{k}^{\mathrm{T}}P_{k}^{\ast}B_{k}%
\end{array}
\right]  +\mathit{\Theta}_{k}\left(  Q_{k}^{\ast}\right) \nonumber \\
=  &  \left[
\begin{array}
[c]{cc}%
A_{k} & B_{k}%
\end{array}
\right]  ^{\mathrm{T}}P_{k}^{\ast}\left[
\begin{array}
[c]{cc}%
A_{k} & B_{k}%
\end{array}
\right]  -L_{k}^{\mathrm{T}}P_{k}^{\ast}L_{k}\nonumber \\
&  +\left[
\begin{array}
[c]{cc}%
\mathscr{A}_{k} & \mathscr{B}_{k}%
\end{array}
\right]  ^{\mathrm{T}}Q_{k}^{\ast}\left[
\begin{array}
[c]{cc}%
\mathscr{A}_{k} & \mathscr{B}_{k}%
\end{array}
\right]  -\left[
\begin{array}
[c]{cc}%
A_{k} & B_{k}%
\end{array}
\right]  ^{\mathrm{T}}Q_{k}^{\ast}\left[
\begin{array}
[c]{cc}%
A_{k} & B_{k}%
\end{array}
\right] \nonumber \\
=  &  \mathit{\Omega}_{k}\left(  P_{k}^{\ast},Q_{k}^{\ast}\right)  .
\label{eq77c}%
\end{align}
By comparing (\ref{eq77c}) with (\ref{eq77b}), we know that the LMI in
(\ref{eq55}) is feasible with $\left(  P_{k},Q_{k}\right)  =\left(
P_{k}^{\ast},Q_{k}^{\ast}\right)  .$ In the following, we will show that
$P_{k}^{\ast}>0$.

Straightforward computation gives that%
\[
W_{k}\left[
\begin{array}
[c]{cc}%
\mathscr{A}_{k} & \mathscr{B}_{k}%
\end{array}
\right]  =\left[
\begin{array}
[c]{ccccc}%
B^{\left[  k\right]  } & B^{\left[  k-1\right]  }A^{\left[  1\right]  } &
\cdots & B^{\left[  1\right]  }A^{\left[  k-1\right]  } & A^{\left[  k\right]
}\\
& B^{\left[  k-1\right]  } & B^{\left[  k-2\right]  }A^{\left[  1\right]  } &
\ddots & A^{\left[  k-1\right]  }\\
&  & \ddots & \ddots & \vdots \\
&  & B^{\left[  2\right]  } & B^{\left[  1\right]  }A^{\left[  1\right]  } &
A^{\left[  2\right]  }\\
&  &  & B^{\left[  1\right]  } & A^{\left[  1\right]  }%
\end{array}
\right]  ,
\]
and%
\[
W_{k}\left[
\begin{array}
[c]{cc}%
A_{k} & B_{k}%
\end{array}
\right]  =\left[
\begin{array}
[c]{cccccc}%
0 & B^{\left[  k-1\right]  } & B^{\left[  k-2\right]  }A^{\left[  1\right]  }
& \cdots & B^{\left[  1\right]  }A^{\left[  k-2\right]  } & A^{\left[
k-1\right]  }\\
& 0 & B^{\left[  k-2\right]  } & B^{\left[  k-3\right]  }A^{\left[  1\right]
} & \ddots & A^{\left[  k-2\right]  }\\
&  & 0 & \ddots & \ddots & \vdots \\
&  &  & \ddots & B^{\left[  1\right]  } & A^{\left[  1\right]  }\\
&  &  &  & 0 & I_{n}%
\end{array}
\right]  .
\]
It follows that we can write%
\begin{align*}
W_{k}A_{k}  &  =\left[
\begin{array}
[c]{cc}%
0_{\left(  k-1\right)  n\times n} & U_{k}\\
0_{n\times n} & 0_{n\times \left(  k-1\right)  n}%
\end{array}
\right]  ,W_{k}B_{k}=\left[
\begin{array}
[c]{c}%
A^{\left[  k-1\right]  }\\
\vdots \\
A^{\left[  1\right]  }\\
I_{n}%
\end{array}
\right]  ,\\
W_{k}\mathscr{A}_{k}  &  =\left[
\begin{array}
[c]{cc}%
B^{\left[  k\right]  } & V_{k}\\
0_{\left(  k-1\right)  n\times n} & U_{k}%
\end{array}
\right]  ,W_{k}\mathscr{B}_{k}=\left[
\begin{array}
[c]{c}%
A^{\left[  k\right]  }\\
\vdots \\
A^{\left[  2\right]  }\\
A^{\left[  1\right]  }%
\end{array}
\right]  ,
\end{align*}
where%
\begin{align*}
U_{k}  &  =\left[
\begin{array}
[c]{cccc}%
B^{\left[  k-1\right]  } & B^{\left[  k-2\right]  }A^{\left[  1\right]  } &
\cdots & B^{\left[  1\right]  }A^{\left[  k-2\right]  }\\
& \ddots & \ddots & \vdots \\
&  & B^{\left[  2\right]  } & B^{\left[  1\right]  }A^{\left[  1\right]  }\\
&  &  & B^{\left[  1\right]  }%
\end{array}
\right]  ,\\
V_{k}  &  =\left[
\begin{array}
[c]{cccc}%
B^{\left[  k-1\right]  }A^{\left[  1\right]  } & B^{\left[  k-2\right]
}A^{\left[  2\right]  } & \cdots & B^{\left[  1\right]  }A^{\left[
k-1\right]  }%
\end{array}
\right]  .
\end{align*}
We also denote%
\begin{align*}
C_{k}  &  =\left[
\begin{array}
[c]{cc}%
B^{\left[  k\right]  } & V_{k}%
\end{array}
\right] \\
&  =\left[
\begin{array}
[c]{cccc}%
B^{\left[  k\right]  } & B^{\left[  k-1\right]  }A^{\left[  1\right]  } &
\cdots & B^{\left[  1\right]  }A^{\left[  k-1\right]  }%
\end{array}
\right]  \in \mathbf{R}^{n\times kn},\; \\
D_{k}  &  =A^{\left[  k\right]  }\in \mathbf{R}^{n\times n}.
\end{align*}
Then, by straightforward computations, we obtain%
\begin{align*}
\mathscr{A}_{k}^{\mathrm{T}}W_{k}^{\mathrm{T}}W_{k}\mathscr{A}_{k}%
-A_{k}^{\mathrm{T}}W_{k}^{\mathrm{T}}W_{k}A_{k}  &  =\left[
\begin{array}
[c]{cc}%
\left(  B^{\left[  k\right]  }\right)  ^{\mathrm{T}}B^{\left[  k\right]  } &
\left(  B^{\left[  k\right]  }\right)  ^{\mathrm{T}}V_{k}\\
V_{k}^{\mathrm{T}}B^{\left[  k\right]  } & V_{k}^{\mathrm{T}}V_{k}%
+U_{k}^{\mathrm{T}}U_{k}%
\end{array}
\right]  -\left[
\begin{array}
[c]{cc}%
0_{n\times n} & 0_{n\times \left(  k-1\right)  n}\\
0_{\left(  k-1\right)  n\times n} & U_{k}^{\mathrm{T}}U_{k}%
\end{array}
\right] \\
&  =\left[
\begin{array}
[c]{cc}%
\left(  B^{\left[  k\right]  }\right)  ^{\mathrm{T}}B^{\left[  k\right]  } &
\left(  B^{\left[  k\right]  }\right)  ^{\mathrm{T}}V_{k}\\
V_{k}^{\mathrm{T}}B^{\left[  k\right]  } & V_{k}^{\mathrm{T}}V_{k}%
\end{array}
\right] \\
&  =C_{k}^{\mathrm{T}}C_{k}.
\end{align*}
Similarly, we have%
\begin{align*}
\mathscr{A}_{k}^{\mathrm{T}}W_{k}^{\mathrm{T}}W_{k}\mathscr{B}_{k}%
-A_{k}^{\mathrm{T}}W_{k}^{\mathrm{T}}W_{k}B_{k}  &  =\left[
\begin{array}
[c]{c}%
\left(  B^{\left[  k\right]  }\right)  ^{\mathrm{T}}A^{\left[  k\right]  }\\
V_{k}^{\mathrm{T}}A^{\left[  k\right]  }+U_{k}^{\mathrm{T}}\left[
\begin{array}
[c]{c}%
A^{\left[  k-1\right]  }\\
\vdots \\
A^{\left[  1\right]  }%
\end{array}
\right]
\end{array}
\right]  -\left[
\begin{array}
[c]{c}%
0_{n\times n}\\
U_{k}^{\mathrm{T}}\left[
\begin{array}
[c]{c}%
A^{\left[  k-1\right]  }\\
\vdots \\
A^{\left[  1\right]  }%
\end{array}
\right]
\end{array}
\right] \\
&  =\left[
\begin{array}
[c]{c}%
\left(  B^{\left[  k\right]  }\right)  ^{\mathrm{T}}A^{\left[  k\right]  }\\
V_{k}^{\mathrm{T}}A^{\left[  k\right]  }%
\end{array}
\right] \\
&  =C_{k}^{\mathrm{T}}D_{k},
\end{align*}
and%
\[
\mathscr{B}_{k}^{\mathrm{T}}W_{k}^{\mathrm{T}}W_{k}\mathscr{B}_{k}%
-B_{k}^{\mathrm{T}}W_{k}^{\mathrm{T}}W_{k}B_{k}=\left(  A^{\left[  k\right]
}\right)  ^{\mathrm{T}}A^{\left[  k\right]  }-I_{n}=D_{k}^{\mathrm{T}}%
D_{k}-I_{n}.
\]
Therefore, we can get%
\begin{equation}
\mathit{\Omega}_{k}\left(  P_{k}^{\ast},Q_{k}^{\ast}\right)  =\left[
\begin{array}
[c]{cc}%
A_{k}^{\mathrm{T}}P_{k}^{\ast}A_{k}-P_{k}^{\ast}+C_{k}^{\mathrm{T}}C_{k} &
A_{k}^{\mathrm{T}}P_{k}B_{k}+C_{k}^{\mathrm{T}}D_{k}\\
B_{k}^{\mathrm{T}}P_{k}^{\ast}A_{k}+D_{k}^{\mathrm{T}}C_{k} & B_{k}%
^{\mathrm{T}}P_{k}^{\ast}B_{k}+D_{k}^{\mathrm{T}}D_{k}-I_{n}%
\end{array}
\right]  , \label{eq887}%
\end{equation}
which, together with (\ref{eq77c}), implies that
\[
A_{k}^{\mathrm{T}}P_{k}^{\ast}A_{k}-P_{k}^{\ast}+C_{k}^{\mathrm{T}}C_{k}<0.
\]
As $A_{k}$ is Schur stable, the above equation implies $P_{k}^{\ast}>0.$

By now we have shown that, if $B$ is nonsingular, the LMI in (\ref{eq55}) is
solvable with positive definite matrices $P_{k}^{\ast}$ and $Q_{k}^{\ast
}=W_{k}^{\mathrm{T}}W_{k}.$ However, if $B$ is singular, the matrix
$Q_{k}^{\ast}=W_{k}^{\mathrm{T}}W_{k}$ is only semi-positive definite. In the
following, we will show that the LMI in (\ref{eq55}) is also feasible with
$\left(  P_{k},Q_{k}\right)  =\left(  P_{k}^{\ast},Q_{k}^{\ast}+\varepsilon
I_{kn}\right)  $ where $\varepsilon>0$ is sufficiently small,$\ $namely,%
\begin{equation}
\mathit{\Omega}_{k}\left(  P_{k}^{\ast},Q_{k}^{\ast}+\varepsilon
I_{kn}\right)  <0. \label{eq774}%
\end{equation}

In fact, it follows from (\ref{eq77b}) that%
\begin{align*}
\mathit{\Omega}_{k}\left(  P_{k}^{\ast},Q_{k}^{\ast}+\varepsilon
I_{kn}\right)   &  =\mathit{\Omega}_{k}\left(  P_{k}^{\ast},Q_{k}^{\ast
}\right)  +\mathit{\Omega}_{k}\left(  0_{kn\times kn},\varepsilon
I_{kn}\right) \\
&  =\mathit{\Omega}_{k}\left(  P_{k}^{\ast},Q_{k}^{\ast}\right)
+\varepsilon \left(  \left[
\begin{array}
[c]{cc}%
\mathscr{A}_{k} & \mathscr{B}_{k}%
\end{array}
\right]  ^{\mathrm{T}}\left[
\begin{array}
[c]{cc}%
\mathscr{A}_{k} & \mathscr{B}_{k}%
\end{array}
\right]  -\left[
\begin{array}
[c]{cc}%
A_{k} & B_{k}%
\end{array}
\right]  ^{\mathrm{T}}\left[
\begin{array}
[c]{cc}%
A_{k} & B_{k}%
\end{array}
\right]  \right) \\
&  \leq \mathit{\Omega}_{k}\left(  P_{k}^{\ast},Q_{k}^{\ast}\right)
+\varepsilon \left(  \left \Vert \left[
\begin{array}
[c]{cc}%
\mathscr{A}_{k} & \mathscr{B}_{k}%
\end{array}
\right]  \right \Vert ^{2}+\left \Vert \left[
\begin{array}
[c]{cc}%
A_{k} & B_{k}%
\end{array}
\right]  \right \Vert ^{2}\right)  .
\end{align*}
Since $\mathit{\Omega}_{k}\left(  P_{k}^{\ast},Q_{k}^{\ast}\right)  $ is
independent of $\varepsilon$ and satisfies (\ref{eq77c}), there exists a
sufficiently small $\varepsilon>0$ such that (\ref{eq774}) is satisfied. The
proof is finished.
\end{proof}

By a Schur complement, the LMI (\ref{eq55}) can be written as%
\[
\left[
\begin{array}
[c]{cc}%
\mathit{\Omega}_{k1}\left(  P_{k},Q_{k}\right)  & [\mathscr{A}_{k}%
,\mathscr{B}_{k}]^{\mathrm{T}}Q_{k}\\
Q_{k}[\mathscr{A}_{k},\mathscr{B}_{k}] & -Q_{k}%
\end{array}
\right]  <0,
\]
whose left hand side is a linear function of $\left(  A,B\right)  .$ Thus, the
most important feature of Theorem \ref{th4}, when compared with the results in
\cite{bliman02mssp} (see Lemma \ref{th2}), the result in \cite{eih06tac} (see
Lemma \ref{th1}) and the method in \cite{hv12auto}, is that the coefficient
$\left(  A,B\right)  $ appears as a linear function. Such a property is
helpful for solving the robust stability analysis problem, as made clear below.

Consider the perturbed system of (\ref{sys})
\begin{equation}
x\left(  t\right)  =\left(  A+\Delta A\right)  x\left(  t-a\right)  +\left(
B+\Delta B\right)  x\left(  t-b\right)  , \label{sys2}%
\end{equation}
where $A\in \mathbf{R}^{n\times n}$ and $B\in \mathbf{R}^{n\times n}$ are the
same as that in (\ref{sys}) and
\begin{equation}
\left[
\begin{array}
[c]{cc}%
\Delta B & \Delta A
\end{array}
\right]  =E_{0}F\left[
\begin{array}
[c]{cc}%
B_{0} & A_{0}%
\end{array}
\right]  , \label{eq884}%
\end{equation}
where $E_{0}\in \mathbf{R}^{n\times p},B_{0}\in \mathbf{R}^{q\times n},A_{0}%
\in \mathbf{R}^{q\times n}$ are known matrices, and $F\in \mathbf{R}^{p\times
q}$ denotes the norm bounded uncertainty (which can be time-varying) that
satisfies%
\begin{equation}
F^{\mathrm{T}}F\leq I_{q}. \label{eqf}%
\end{equation}
For further using, we denote
\[
\left[
\begin{array}
[c]{cc}%
\mathscr{A}_{k0} & \mathscr{B}_{k0}%
\end{array}
\right]  =\left[
\begin{array}
[c]{ccccc}%
B_{0} & A_{0} & 0 & \cdots & 0\\
& \ddots & \ddots & \ddots & \vdots \\
&  & B_{0} & A_{0} & 0\\
&  &  & B_{0} & A_{0}%
\end{array}
\right]  \in \mathbf{R}^{kq\times \left(  k+1\right)  n}.
\]

\begin{theorem}
\label{th5}The uncertain linear difference equation (\ref{sys2}) is
exponentially stable for any $F\in \mathbf{R}^{p\times q}$ satisfying
(\ref{eqf}) if there exists an integer $k\geq1,$ positive definite matrices
$P_{k},Q_{k}\in \mathbf{R}^{kn\times kn}$ and a positive definite matrix
$S_{k}\in \mathbf{R}^{k\times k}$ such that the following LMI is satisfied:%
\begin{equation}
\left[
\begin{array}
[c]{cc}%
\mathit{\Omega}_{k}\left(  P_{k},Q_{k}\right)  +\left[
\begin{array}
[c]{cc}%
\mathscr{A}_{k0} & \mathscr{B}_{k0}%
\end{array}
\right]  ^{\mathrm{T}}\left(  S_{k}\otimes I_{q}\right)  \left[
\begin{array}
[c]{cc}%
\mathscr{A}_{k0} & \mathscr{B}_{k0}%
\end{array}
\right]  & \left[
\begin{array}
[c]{cc}%
\mathscr{A}_{k0} & \mathscr{B}_{k0}%
\end{array}
\right]  ^{\mathrm{T}}Q_{k}\left(  I_{k}\otimes E_{0}\right) \\
\left(  I_{k}\otimes E_{0}^{\mathrm{T}}\right)  Q_{k}\left[
\begin{array}
[c]{cc}%
\mathscr{A}_{k} & \mathscr{B}_{k}%
\end{array}
\right]  & \left(  I_{k}\otimes E_{0}^{\mathrm{T}}\right)  Q_{k}\left(
I_{k}\otimes E_{0}\right)  -S_{k}\otimes I_{p}%
\end{array}
\right]  <0. \label{eq88}%
\end{equation}

\end{theorem}

\begin{proof}
For notation simplicity, we denote
\begin{equation}
\mathscr{C}_{k}=\left[
\begin{array}
[c]{cc}%
\mathscr{A}_{k} & \mathscr{B}_{k}%
\end{array}
\right]  ,\; \mathscr{C}_{k0}=\left[
\begin{array}
[c]{cc}%
\mathscr{A}_{k0} & \mathscr{B}_{k0}%
\end{array}
\right]  ,\nonumber
\end{equation}
$\mathit{\Omega}_{k}=\mathit{\Omega}_{k}\left(  P_{k},Q_{k}\right)  ,$ and
$\mathit{\Omega}_{k1}=\mathit{\Omega}_{k1}\left(  P_{k},Q_{k}\right)  .$
Notice that we can write%
\begin{align*}
0  &  >\left[
\begin{array}
[c]{cc}%
\mathit{\Omega}_{k}+\mathscr{C}_{k0}^{\mathrm{T}}\left(  S_{k}\otimes
I_{q}\right)  \mathscr{C}_{k0} & \mathscr{C}_{k}^{\mathrm{T}}Q_{k}\left(
I_{k}\otimes E_{0}\right) \\
\left(  I_{k}\otimes E_{0}^{\mathrm{T}}\right)  Q_{k}\mathscr{C}_{k} & \left(
I_{k}\otimes E_{0}^{\mathrm{T}}\right)  Q_{k}\left(  I_{k}\otimes
E_{0}\right)  -S_{k}\otimes I_{p}%
\end{array}
\right] \\
&  =\left[
\begin{array}
[c]{cc}%
\mathit{\Omega}_{k1}+\mathscr{C}_{k}^{\mathrm{T}}Q_{k}\mathscr{C}_{k}%
+\mathscr{C}_{k0}^{\mathrm{T}}\left(  S_{k}\otimes I_{q}\right)
\mathscr{C}_{k0} & \mathscr{C}_{k}^{\mathrm{T}}Q_{k}\left(  I_{k}\otimes
E_{0}\right) \\
\left(  I_{k}\otimes E_{0}^{\mathrm{T}}\right)  Q_{k}\mathscr{C}_{k} & \left(
I_{k}\otimes E_{0}^{\mathrm{T}}\right)  Q_{k}\left(  I_{k}\otimes
E_{0}\right)  -S_{k}\otimes I_{p}%
\end{array}
\right] \\
&  =\left[
\begin{array}
[c]{cc}%
\mathit{\Omega}_{k1}+\mathscr{C}_{k0}^{\mathrm{T}}\left(  S_{k}\otimes
I_{q}\right)  \mathscr{C}_{k0} & 0_{\left(  k+1\right)  n\times kp}\\
0_{kp\times \left(  k+1\right)  n} & -S_{k}\otimes I_{p}%
\end{array}
\right]  +\left[
\begin{array}
[c]{c}%
\mathscr{C}_{k}^{\mathrm{T}}Q_{k}\\
\left(  I_{k}\otimes E_{0}^{\mathrm{T}}\right)  Q_{k}%
\end{array}
\right]  Q_{k}^{-1}\left[
\begin{array}
[c]{c}%
\mathscr{C}_{k}^{\mathrm{T}}Q_{k}\\
\left(  I_{k}\otimes E_{0}^{\mathrm{T}}\right)  Q_{k}%
\end{array}
\right]  ^{\mathrm{T}},
\end{align*}
which, by a Schur complement, is equivalent to%
\[
\left[
\begin{array}
[c]{ccc}%
\mathit{\Omega}_{k1}+\mathscr{C}_{k0}^{\mathrm{T}}\left(  S_{k}\otimes
I_{q}\right)  \mathscr{C}_{k0} & 0_{\left(  k+1\right)  n\times kp} &
\mathscr{C}_{k}^{\mathrm{T}}Q_{k}\\
0_{kp\times \left(  k+1\right)  n} & -S_{k}\otimes I_{p} & \left(  I_{k}\otimes
E_{0}^{\mathrm{T}}\right)  Q_{k}\\
Q_{k}\mathscr{C}_{k} & Q_{k}\left(  I_{k}\otimes E_{0}\right)  & -Q_{k}%
\end{array}
\right]  <0.
\]
By a congruence transformation, this is equivalent to%
\[
\left[
\begin{array}
[c]{ccc}%
\mathit{\Omega}_{k1}+\mathscr{C}_{k0}^{\mathrm{T}}\left(  S_{k}\otimes
I_{q}\right)  \mathscr{C}_{k0} & \mathscr{C}_{k}^{\mathrm{T}}Q_{k} &
0_{\left(  k+1\right)  n\times kp}\\
Q_{k}\mathscr{C}_{k} & -Q_{k} & Q_{k}\left(  I_{k}\otimes E_{0}\right) \\
0_{kp\times \left(  k+1\right)  n} & \left(  I_{k}\otimes E_{0}^{\mathrm{T}%
}\right)  Q_{k} & -S_{k}\otimes I_{p}%
\end{array}
\right]  <0.
\]
By a Schur complement, the above inequality holds true if and only if%
\begin{align}
0>  &  \left[
\begin{array}
[c]{cc}%
\mathit{\Omega}_{k1} & \mathscr{C}_{k}^{\mathrm{T}}Q_{k}\\
Q_{k}\mathscr{C}_{k} & -Q_{k}%
\end{array}
\right]  +\left[
\begin{array}
[c]{c}%
0_{\left(  k+1\right)  n\times kp}\\
Q_{k}\left(  I_{k}\otimes E_{0}\right)
\end{array}
\right]  \left(  S_{k}^{-1}\otimes I_{p}\right)  \left[
\begin{array}
[c]{c}%
0_{\left(  k+1\right)  n\times kp}\\
Q_{k}\left(  I_{k}\otimes E_{0}\right)
\end{array}
\right]  ^{\mathrm{T}}\nonumber \\
&  +\left[
\begin{array}
[c]{cc}%
\mathscr{C}_{k0} & 0_{kq\times kn}%
\end{array}
\right]  ^{\mathrm{T}}\left(  S_{k}\otimes I_{q}\right)  \left[
\begin{array}
[c]{cc}%
\mathscr{C}_{k0} & 0_{kq\times kn}%
\end{array}
\right]  . \label{eq885}%
\end{align}
By (\ref{eq884}) we have%
\begin{align*}
\Delta \mathscr{C}_{k}  &  \triangleq \left[
\begin{array}
[c]{cccccc}%
\Delta B & \Delta A & 0 & \cdots & 0 & 0\\
& \Delta B & \Delta A & \ddots & \vdots & \vdots \\
&  & \ddots & \ddots & 0 & 0\\
&  &  & \Delta B & \Delta A & 0\\
&  &  &  & \Delta B & \Delta A
\end{array}
\right] \\
&  =\left[
\begin{array}
[c]{cccccc}%
E_{0}FB_{0} & E_{0}FA_{0} & 0 & \cdots & 0 & 0\\
& E_{0}FB_{0} & E_{0}FA_{0} & \ddots & \vdots & 0\\
&  & \ddots & \ddots & 0 & \vdots \\
&  &  & E_{0}FB_{0} & E_{0}FA_{0} & 0\\
&  &  &  & E_{0}FB_{0} & E_{0}FA_{0}%
\end{array}
\right] \\
&  =\left(  I_{k}\otimes E_{0}\right)  \left(  I_{k}\otimes F\right)
\mathscr{C}_{k0}.
\end{align*}
By using (\ref{eqf}) we can compute%
\[
\left(  I_{k}\otimes F^{\mathrm{T}}\right)  \left(  S_{k}\otimes I_{p}\right)
\left(  I_{k}\otimes F\right)  =S_{k}\otimes F^{\mathrm{T}}F\leq S_{k}\otimes
I_{q}.
\]
Therefore, by using Lemma \ref{lm1}, we have from (\ref{eq885}) that%
\begin{align}
&  \left[
\begin{array}
[c]{cc}%
\mathit{\Omega}_{k1} & \left(  \mathscr{C}_{k}+\Delta \mathscr{C}_{k}\right)
^{\mathrm{T}}Q_{k}\\
Q_{k}\left(  \mathscr{C}_{k}+\Delta \mathscr{C}_{k}\right)  & -Q_{k}%
\end{array}
\right] \nonumber \\
=  &  \left[
\begin{array}
[c]{cc}%
\mathit{\Omega}_{k1} & \mathscr{C}_{k}^{\mathrm{T}}Q_{k}\\
Q_{k}\mathscr{C}_{k} & -Q_{k}%
\end{array}
\right]  +\left[
\begin{array}
[c]{cc}%
0_{\left(  k+1\right)  n\times \left(  k+1\right)  n} & \Delta \mathscr{C}_{k}%
^{\mathrm{T}}Q_{k}\\
Q_{k}\Delta \mathscr{C}_{k} & 0_{kn\times kn}%
\end{array}
\right] \nonumber \\
=  &  \left[
\begin{array}
[c]{cc}%
\mathit{\Omega}_{k1} & \mathscr{C}_{k}^{\mathrm{T}}Q_{k}\\
Q_{k}\mathscr{C}_{k} & -Q_{k}%
\end{array}
\right]  +\left[
\begin{array}
[c]{c}%
0_{\left(  k+1\right)  n\times kp}\\
Q_{k}\left(  I_{k}\otimes E_{0}\right)
\end{array}
\right]  \left(  I_{k}\otimes F\right)  \left[
\begin{array}
[c]{cc}%
\mathscr{C}_{k0} & 0_{kq\times kn}%
\end{array}
\right] \nonumber \\
&  +\left[
\begin{array}
[c]{cc}%
\mathscr{C}_{k0} & 0_{kq\times kn}%
\end{array}
\right]  ^{\mathrm{T}}\left(  I_{k}\otimes F^{\mathrm{T}}\right)  \left[
\begin{array}
[c]{c}%
0_{\left(  k+1\right)  n\times kp}\\
Q_{k}\left(  I_{k}\otimes E_{0}\right)
\end{array}
\right]  ^{\mathrm{T}}\nonumber \\
\leq &  \left[
\begin{array}
[c]{cc}%
\mathit{\Omega}_{k1} & \mathscr{C}_{k}^{\mathrm{T}}Q_{k}\\
Q_{k}\mathscr{C}_{k} & -Q_{k}%
\end{array}
\right]  +\left[
\begin{array}
[c]{c}%
0_{\left(  k+1\right)  n\times kp}\\
Q_{k}\left(  I_{k}\otimes E_{0}\right)
\end{array}
\right]  \left(  S_{k}^{-1}\otimes I_{p}\right)  \left[
\begin{array}
[c]{c}%
0_{\left(  k+1\right)  n\times kp}\\
Q_{k}\left(  I_{k}\otimes E_{0}\right)
\end{array}
\right]  ^{\mathrm{T}}\nonumber \\
&  +\left[
\begin{array}
[c]{cc}%
\mathscr{C}_{k0} & 0_{kq\times kn}%
\end{array}
\right]  ^{\mathrm{T}}\left(  I_{k}\otimes F^{\mathrm{T}}\right)  \left(
S_{k}\otimes I_{p}\right)  \left(  I_{k}\otimes F\right)  \left[
\begin{array}
[c]{cc}%
\mathscr{C}_{k0} & 0_{kq\times kn}%
\end{array}
\right] \nonumber \\
\leq &  \left[
\begin{array}
[c]{cc}%
\mathit{\Omega}_{k1} & \mathscr{C}_{k}^{\mathrm{T}}Q_{k}\\
Q_{k}\mathscr{C}_{k} & -Q_{k}%
\end{array}
\right]  +\left[
\begin{array}
[c]{c}%
0_{\left(  k+1\right)  n\times kp}\\
Q_{k}\left(  I_{k}\otimes E_{0}\right)
\end{array}
\right]  \left(  S_{k}^{-1}\otimes I_{p}\right)  \left[
\begin{array}
[c]{c}%
0_{\left(  k+1\right)  n\times kp}\\
Q_{k}\left(  I_{k}\otimes E_{0}\right)
\end{array}
\right]  ^{\mathrm{T}}\nonumber \\
&  +\left[
\begin{array}
[c]{cc}%
\mathscr{C}_{k0} & 0_{kq\times kn}%
\end{array}
\right]  ^{\mathrm{T}}\left(  S_{k}\otimes I_{q}\right)  \left[
\begin{array}
[c]{cc}%
\mathscr{C}_{k0} & 0_{kq\times kn}%
\end{array}
\right] \nonumber \\
<  &  0. \label{eq889}%
\end{align}
By a Schur complement, the above inequality is equivalent to%
\begin{align*}
0>  &  \mathit{\Omega}_{k1}+\left(  \mathscr{C}_{k}+\Delta \mathscr{C}_{k}%
\right)  ^{\mathrm{T}}Q_{k}\left(  \mathscr{C}_{k}+\Delta \mathscr{C}_{k}%
\right) \\
=  &  \left[
\begin{array}
[c]{cc}%
A_{k} & B_{k}%
\end{array}
\right]  ^{\mathrm{T}}\left(  P_{k}-Q_{k}\right)  \left[
\begin{array}
[c]{cc}%
A_{k} & B_{k}%
\end{array}
\right]  -L_{k}^{\mathrm{T}}P_{k}L_{k}\\
&  +\left[
\begin{array}
[c]{cc}%
\mathscr{A}_{k}+\Delta \mathscr{A}_{k} & \mathscr{B}_{k}+\Delta \mathscr{B}_{k}%
\end{array}
\right]  ^{\mathrm{T}}Q_{k}\left[
\begin{array}
[c]{cc}%
\mathscr{A}_{k}+\Delta \mathscr{A}_{k} & \mathscr{B}_{k}+\Delta \mathscr{B}_{k}%
\end{array}
\right]  .
\end{align*}
By Theorem \ref{th4}, we know that system (\ref{sys2}) is exponentially
stable. The proof is finished.
\end{proof}

The merit of the proof of Theorem \ref{th5} is that we have utilized the fact
that $\left(  A,B\right)  $ appears as a linear function in the LMIs, which
helps to eliminate the uncertain matrix $F$ in the LMI (\ref{eq55}). This can
not be achieved for the LMI in Lemmas \ref{th1} and \ref{th2}. Moreover, from
the proof we can see that the only conservatism comes from the usage of the
inequality in Lemma \ref{lm1}. Thus the condition in Theorem \ref{th5} is
considered to be quite tight.

By using again the property that $\left(  A,B\right)  $ appears in the matrix
$\mathit{\Omega}_{k}\left(  P_{k},Q_{k}\right)  $ as a quadratic function, we
can extend easily the results in Theorem \ref{th5} to the case of polytopic
type uncertainty, say,%
\begin{equation}
\left[
\begin{array}
[c]{cc}%
\Delta A & \Delta B
\end{array}
\right]  \in \mathrm{co}\left \{  \left[
\begin{array}
[c]{cc}%
A^{\left(  i\right)  } & B^{\left(  i\right)  }%
\end{array}
\right]  ,i=1,2,\ldots,N\right \}  , \label{eq886}%
\end{equation}
where $A^{\left(  i\right)  },B^{\left(  i\right)  },i=1,2,\ldots,N$ are given
matrices. Denote
\[
\left[
\begin{array}
[c]{cc}%
\mathscr{A}_{k}^{\left(  i\right)  } & \mathscr{B}_{k}^{\left(  i\right)  }%
\end{array}
\right]  =\left[
\begin{array}
[c]{ccccc}%
B+B^{\left(  i\right)  } & A+A^{\left(  i\right)  } & \cdots & 0 & 0\\
&  & \ddots & 0 & \vdots \\
&  & B+B^{\left(  i\right)  } & A+A^{\left(  i\right)  } & 0\\
&  &  & B+B^{\left(  i\right)  } & A+A^{\left(  i\right)  }%
\end{array}
\right]  \in \mathbf{R}^{kn\times \left(  k+1\right)  n}.
\]
Then we obtain immediately the following theorem.

\begin{theorem}
\label{th6}The uncertain linear difference equation (\ref{sys2}), where
$\Delta A$ and $\Delta B$ satisfy (\ref{eq886}), is exponentially stable if
there exists positive definite matrices $P_{k},Q_{k}\in \mathbf{R}^{kn\times
kn}$ such that
\begin{align}
\mathit{\Omega}_{k}^{\left(  i\right)  }\left(  P_{k},Q_{k}\right)  =  &
\mathit{\Omega}_{k1}\left(  P_{k},Q_{k}\right)  +\left[
\begin{array}
[c]{cc}%
\mathscr{A}_{k}^{\left(  i\right)  } & \mathscr{B}_{k}^{\left(  i\right)  }%
\end{array}
\right]  ^{\mathrm{T}}Q_{k}\left[
\begin{array}
[c]{cc}%
\mathscr{A}_{k}^{\left(  i\right)  } & \mathscr{B}_{k}^{\left(  i\right)  }%
\end{array}
\right] \nonumber \\
<  &  0, \label{eq888}%
\end{align}
are satisfied for $i=1,2,\ldots,N.$
\end{theorem}

\begin{proof}
Notice that (\ref{eq888}) implies%
\[
\left[
\begin{array}
[c]{cc}%
\mathit{\Omega}_{k1}\left(  P_{k},Q_{k}\right)  & \left[
\begin{array}
[c]{cc}%
\mathscr{A}_{k}^{\left(  i\right)  } & \mathscr{B}_{k}^{\left(  i\right)  }%
\end{array}
\right]  ^{\mathrm{T}}Q_{k}\\
Q_{k}\left[
\begin{array}
[c]{cc}%
\mathscr{A}_{k}^{\left(  i\right)  } & \mathscr{B}_{k}^{\left(  i\right)  }%
\end{array}
\right]  & -Q_{k}%
\end{array}
\right]  <0,
\]
where $i=1,2,\ldots,N$. It follows that, for any $\alpha_{i}\geq
0,i=1,2,\ldots,N$ with $\alpha_{1}+\alpha_{2}+\cdots+\alpha_{N}=1,$ and%
\[
\left[
\begin{array}
[c]{cc}%
\Delta A & \Delta B
\end{array}
\right]  =\sum \limits_{i=1}^{N}\alpha_{i}\left[
\begin{array}
[c]{cc}%
A^{\left(  i\right)  } & B^{\left(  i\right)  }%
\end{array}
\right]  ,
\]
we have
\begin{align*}
0  &  >\left[
\begin{array}
[c]{cc}%
\sum \limits_{i=1}^{N}\alpha_{i}\mathit{\Omega}_{k1}\left(  P_{k},Q_{k}\right)
& \sum \limits_{i=1}^{N}\alpha_{i}\left[
\begin{array}
[c]{cc}%
\mathscr{A}_{k}^{\left(  i\right)  } & \mathscr{B}_{k}^{\left(  i\right)  }%
\end{array}
\right]  ^{\mathrm{T}}Q_{k}\\
Q_{k}\sum \limits_{i=1}^{N}\alpha_{i}\left[
\begin{array}
[c]{cc}%
\mathscr{A}_{k}^{\left(  i\right)  } & \mathscr{B}_{k}^{\left(  i\right)  }%
\end{array}
\right]  & -\sum \limits_{i=1}^{N}\alpha_{i}Q_{k}%
\end{array}
\right] \\
&  =\left[
\begin{array}
[c]{cc}%
\mathit{\Omega}_{k1}\left(  P_{k},Q_{k}\right)  & \left(  \mathscr{C}_{k}%
+\Delta \mathscr{C}_{k}\right)  ^{\mathrm{T}}Q_{k}\\
Q_{k}\left(  \mathscr{C}_{k}+\Delta \mathscr{C}_{k}\right)  & -Q_{k}%
\end{array}
\right]  ,
\end{align*}
which is exactly in the form of (\ref{eq889}). The remaining of the proof is
similar to that of Theorem \ref{th5} and is omitted.
\end{proof}

\section{\label{sec3}Interpretations and Relationships}

We first provide time-domain interpretations of Theorem \ref{th4} and Lemma
\ref{th2} by establishing LKFs.

\begin{lemma}
\label{lm5}For any integer $k\geq1,$ there holds%
\begin{equation}
x(t)=\sum \limits_{i=0}^{k}A^{\left[  i\right]  }B^{\left[  k-i\right]
}x\left(  t-ia-\left(  k-i\right)  b\right)  . \label{eq82}%
\end{equation}

\end{lemma}

\begin{proof}
Clearly, it follows from (\ref{sys}) that (\ref{eq82}) holds true with $k=1.$
Assume that (\ref{eq82}) is true with $k=m,$ namely,%
\begin{equation}
x(t)=\sum \limits_{i=0}^{m}A^{\left[  i\right]  }B^{\left[  m-i\right]
}x\left(  t-ia-\left(  m-i\right)  b\right)  . \label{eq81}%
\end{equation}
Then, by inserting (\ref{sys}) into (\ref{eq81}), we have%
\begin{align}
x(t)=  &  \sum \limits_{i=0}^{m}A^{\left[  i\right]  }B^{\left[  m-i\right]
}\left(  Ax\left(  t-\left(  i+1\right)  a-\left(  m-i\right)  b\right)
+Bx\left(  t-ia-\left(  m+1-i\right)  b\right)  \right) \nonumber \\
=  &  \sum \limits_{i=0}^{m}A^{\left[  i\right]  }B^{\left[  m-i\right]
}Bx\left(  t-ia-\left(  m+1-i\right)  b\right)  +\sum \limits_{i=0}%
^{m}A^{\left[  i\right]  }B^{\left[  m-i\right]  }Ax\left(  t-\left(
i+1\right)  a-\left(  m-i\right)  b\right) \nonumber \\
=  &  \sum \limits_{i=0}^{m}A^{\left[  i\right]  }B^{\left[  m-i\right]
}Bx\left(  t-ia-\left(  m+1-i\right)  b\right)  +\sum \limits_{j=1}%
^{m+1}A^{\left[  j-1\right]  }B^{\left[  m+1-j\right]  }Ax\left(  t-ja-\left(
m+1-j\right)  b\right) \nonumber \\
=  &  A^{\left[  0\right]  }B^{\left[  m\right]  }Bx\left(  t-\left(
m+1\right)  b\right)  +\sum \limits_{i=1}^{m}A^{\left[  i\right]  }B^{\left[
m-i\right]  }Bx\left(  t-ia-\left(  m+1-i\right)  b\right) \nonumber \\
&  +\sum \limits_{j=1}^{m}A^{\left[  j-1\right]  }B^{\left[  m+1-j\right]
}Ax\left(  t-ja-\left(  m+1-j\right)  b\right)  +A^{\left[  m\right]
}B^{\left[  0\right]  }Ax\left(  t-\left(  m+1\right)  a\right) \nonumber \\
=  &  B^{\left[  m+1\right]  }x\left(  t-\left(  m+1\right)  b\right)
+A^{\left[  m+1\right]  }x\left(  t-\left(  m+1\right)  a\right) \nonumber \\
&  +\sum \limits_{i=1}^{m}\left(  A^{\left[  i\right]  }B^{\left[  m-i\right]
}B+A^{\left[  i-1\right]  }B^{\left[  m+1-i\right]  }A\right)  x\left(
t-ia-\left(  m+1-i\right)  b\right)  . \label{eq83}%
\end{align}
Notice that (see (\ref{eq87}) in Appendix A2)%
\[
A^{\left[  i\right]  }B^{\left[  m-i\right]  }B+A^{\left[  i-1\right]
}B^{\left[  m+1-i\right]  }A=A^{\left[  i\right]  }B^{\left[  m+1-i\right]
},i=1,2,\ldots,m,
\]
substitution of which into (\ref{eq83}) gives%
\begin{align*}
x(t)  &  =B^{\left[  m+1\right]  }x\left(  t-\left(  m+1\right)  b\right)
+\sum \limits_{i=1}^{m}A^{\left[  i\right]  }B^{\left[  m+1-i\right]  }x\left(
t-ia-\left(  m+1-i\right)  b\right)  +A^{\left[  m+1\right]  }x\left(
t-\left(  m+1\right)  a\right) \\
&  =\sum \limits_{i=0}^{m+1}A^{\left[  i\right]  }B^{\left[  m+1-i\right]
}x\left(  t-ia-\left(  m+1-i\right)  b\right)  .
\end{align*}
Therefore, (\ref{eq82}) holds with $k=m+1.$ The proof is finished by
mathematical induction.
\end{proof}

In the following, we assume, without loss of generality, that $b>a$ since
otherwise we can change the roles of $a$ and $b.$

\begin{lemma}
\label{lm2}For any integer $k\geq1,$ let
\begin{equation}
\left \{
\begin{array}
[c]{rl}%
X_{k}\left(  t\right)  & =\left[
\begin{array}
[c]{c}%
x\left(  t-kb\right) \\
x\left(  t-\left(  k-1\right)  b-a\right) \\
\vdots \\
x\left(  t-2b-\left(  k-2\right)  a\right) \\
x\left(  t-b-\left(  k-1\right)  a\right)
\end{array}
\right]  \in \mathbf{R}^{kn},\\
U_{k}\left(  t\right)  & =x\left(  t-ka\right)  \in \mathbf{R}^{n},\\
Y_{k}\left(  t\right)  & =x\left(  t\right)  \in \mathbf{R}^{n}.
\end{array}
\right.  \label{eq85}%
\end{equation}
Then $\left(  U_{k}\left(  t\right)  ,X_{k}\left(  t\right)  ,Y_{k}\left(
t\right)  \right)  $ satisfies%
\begin{equation}
\left.
\begin{array}
[c]{rl}%
X_{k}\left(  t+b-a\right)  & =A_{k}X_{k}\left(  t\right)  +B_{k}U_{k}\left(
t\right)  ,\\
Y_{k}\left(  t\right)  & =C_{k}X_{k}\left(  t\right)  +D_{k}U_{k}\left(
t\right)  .
\end{array}
\right \}  \label{eq84}%
\end{equation}

\end{lemma}

\begin{proof}
This can be verified by direct computation. In fact, by definition, we have%
\begin{align*}
X_{k}\left(  t+b-a\right)   &  =\left[
\begin{array}
[c]{c}%
x\left(  t-\left(  k-1\right)  b-a\right) \\
x\left(  t-\left(  k-2\right)  b-2a\right) \\
\vdots \\
x\left(  t-b-\left(  k-1\right)  a\right) \\
x\left(  t-ka\right)
\end{array}
\right] \\
&  =A_{k}X_{k}\left(  t\right)  +B_{k}U\left(  t\right)  ,
\end{align*}
and it follows from Lemma \ref{lm5} that%
\begin{align*}
Y_{k}\left(  t\right)   &  =\sum \limits_{i=0}^{k}A^{\left[  i\right]
}B^{\left[  k-i\right]  }x\left(  t-ia-\left(  k-i\right)  b\right) \\
&  =\left[
\begin{array}
[c]{cccc}%
B^{\left[  k\right]  } & A^{\left[  1\right]  }B^{\left[  k-1\right]  } &
\cdots & A^{\left[  k-1\right]  }B^{\left[  1\right]  }%
\end{array}
\right]  X_{k}\left(  t\right)  +A^{\left[  k\right]  }x\left(  t-ka\right) \\
&  =C_{k}X_{k}\left(  t\right)  +D_{k}U_{k}\left(  t\right)  .
\end{align*}
The proof is finished.
\end{proof}

We next provide a time-domain interpretation of Theorem \ref{th4} by
establishing an LKF for the system.

\begin{proposition}
\label{pp1}For any integer $k\geq1,$\ let $\mathit{\Omega}_{k}\left(
P,Q\right)  $ be defined by (\ref{eq77b}) where $\left(  A_{k},B_{k}%
,\mathscr{A}_{k},\mathscr{B}_{k}\right)  $ is defined by (\ref{eqab3}%
)-(\ref{eqab4}). Consider the following LKF
\begin{equation}
V_{k}\left(  x_{t}\right)  =\int_{t-b}^{t-a}X_{k}^{\mathrm{T}}\left(
s\right)  P_{k}X_{k}\left(  s\right)  \mathrm{d}s+\int_{t-a}^{t}%
X_{k}^{\mathrm{T}}\left(  s\right)  Q_{k}X_{k}\left(  s\right)  \mathrm{d}s,
\label{eqv1}%
\end{equation}
where $P_{k}=P_{k}^{\mathrm{T}}\in \mathbf{R}^{kn\times kn}$ and $Q_{k}%
=Q_{k}^{\mathrm{T}}\in \mathbf{R}^{kn\times kn}.$ Then%
\begin{equation}
\dot{V}_{k}\left(  x_{t}\right)  =\left[
\begin{array}
[c]{c}%
X_{k}\left(  t-b\right) \\
x\left(  t-b-ka\right)
\end{array}
\right]  ^{\mathrm{T}}\mathit{\Omega}_{k}\left(  P_{k},Q_{k}\right)  \left[
\begin{array}
[c]{c}%
X_{k}\left(  t-b\right) \\
x\left(  t-b-ka\right)
\end{array}
\right]  . \label{eqdv1}%
\end{equation}

\end{proposition}

\begin{proof}
From (\ref{eq85}) and (\ref{eq84}) we know $U_{k}\left(  t+a-b\right)
=x\left(  t-\left(  k-1\right)  a-b\right)  $ and%
\[
X_{k}\left(  t-a\right)  =A_{k}X_{k}\left(  t-b\right)  +B_{k}x\left(
t-ka-b\right)  .
\]
By using (\ref{sys}) and noting the structures of $\mathscr{A}_{k}$ and
$\mathscr{B}_{k},$ we have%
\begin{align*}
X_{k}\left(  t\right)   &  =\left[
\begin{array}
[c]{c}%
x\left(  t-kb\right) \\
x\left(  t-\left(  k-1\right)  b-a\right) \\
\vdots \\
x\left(  t-2b-\left(  k-2\right)  a\right) \\
x\left(  t-b-\left(  k-1\right)  a\right)
\end{array}
\right] \\
&  =\left[
\begin{array}
[c]{cccccc}%
B & A & 0 & \cdots & 0 & 0\\
& B & A & \ddots & \vdots & 0\\
&  & \ddots & \ddots & 0 & \vdots \\
&  &  & B & A & 0\\
&  &  &  & B & A
\end{array}
\right]  \left[
\begin{array}
[c]{c}%
x\left(  t-\left(  k+1\right)  b\right) \\
x\left(  t-kb-a\right) \\
\vdots \\
x\left(  t-2b-\left(  k-1\right)  a\right) \\
x\left(  t-b-ka\right)
\end{array}
\right] \\
&  =\left[
\begin{array}
[c]{cc}%
\mathscr{A}_{k} & \mathscr{B}_{k}%
\end{array}
\right]  \left[
\begin{array}
[c]{c}%
X_{k}\left(  t-b\right) \\
x\left(  t-b-ka\right)
\end{array}
\right]  .
\end{align*}
Therefore, it follows from (\ref{eq86}) that%
\begin{align*}
\dot{V}_{k}\left(  x_{t}\right)  =  &  X_{k}^{\mathrm{T}}\left(  t-a\right)
P_{k}X_{k}\left(  t-a\right)  -X_{k}^{\mathrm{T}}\left(  t-b\right)
P_{k}X_{k}\left(  t-b\right) \\
&  +X_{k}^{\mathrm{T}}\left(  t\right)  Q_{k}X_{k}\left(  t\right)
-X_{k}^{\mathrm{T}}\left(  t-a\right)  Q_{k}X_{k}\left(  t-a\right) \\
=  &  \left(  A_{k}X_{k}\left(  t-b\right)  +B_{k}x\left(  t-b-ka\right)
\right)  ^{\mathrm{T}}\left(  P_{k}-Q_{k}\right)  \left(  A_{k}X_{k}\left(
t-b\right)  +B_{k}x\left(  t-b-ka\right)  \right) \\
&  +X_{k}^{\mathrm{T}}\left(  t\right)  Q_{k}X_{k}\left(  t\right)
-X_{k}^{\mathrm{T}}\left(  t-b\right)  P_{k}X_{k}\left(  t-b\right) \\
=  &  \left[
\begin{array}
[c]{c}%
X_{k}\left(  t-b\right) \\
x\left(  t-b-ka\right)
\end{array}
\right]  ^{\mathrm{T}}\left[
\begin{array}
[c]{cc}%
A_{k} & B_{k}%
\end{array}
\right]  ^{\mathrm{T}}\left(  P_{k}-Q_{k}\right)  \left[
\begin{array}
[c]{cc}%
A_{k} & B_{k}%
\end{array}
\right]  \left[
\begin{array}
[c]{c}%
X_{k}\left(  t-b\right) \\
x\left(  t-b-ka\right)
\end{array}
\right] \\
&  -\left[
\begin{array}
[c]{c}%
X_{k}\left(  t-b\right) \\
x\left(  t-b-ka\right)
\end{array}
\right]  ^{\mathrm{T}}\left[
\begin{array}
[c]{cc}%
I_{kn} & 0_{kn\times n}%
\end{array}
\right]  ^{\mathrm{T}}P_{k}\left[
\begin{array}
[c]{cc}%
I_{kn} & 0_{kn\times n}%
\end{array}
\right]  \left[
\begin{array}
[c]{c}%
X_{k}\left(  t-b\right) \\
x\left(  t-b-ka\right)
\end{array}
\right] \\
&  +\left[
\begin{array}
[c]{c}%
X_{k}\left(  t-b\right) \\
x\left(  t-b-ka\right)
\end{array}
\right]  ^{\mathrm{T}}\left[
\begin{array}
[c]{cc}%
\mathscr{A}_{k} & \mathscr{B}_{k}%
\end{array}
\right]  ^{\mathrm{T}}Q_{k}\left[
\begin{array}
[c]{cc}%
\mathscr{A}_{k} & \mathscr{B}_{k}%
\end{array}
\right]  \left[
\begin{array}
[c]{c}%
X_{k}\left(  t-b\right) \\
x\left(  t-b-ka\right)
\end{array}
\right] \\
=  &  \left[
\begin{array}
[c]{c}%
X_{k}\left(  t-b\right) \\
x\left(  t-b-ka\right)
\end{array}
\right]  ^{\mathrm{T}}\mathit{\Omega}_{k}\left(  P_{k},Q_{k}\right)  \left[
\begin{array}
[c]{c}%
X_{k}\left(  t-b\right) \\
x\left(  t-b-ka\right)
\end{array}
\right]  .
\end{align*}
The proof is finished.
\end{proof}

Similar to Lemma \ref{lm2}, we can present the following lemma.

\begin{lemma}
\label{lm10}For any integer $k\geq1,$\ let $\overline{A}_{k},\overline{B}_{k}$
be defined in (\ref{eqab1}) and
\begin{equation}
\overline{C}_{k}=\left[
\begin{array}
[c]{cccc}%
B & AB & \cdots & A^{k-1}B
\end{array}
\right]  \in \mathbf{R}^{n\times kn},\; \overline{D}_{k}=A^{k}\in
\mathbf{R}^{n\times n}. \label{eqcd}%
\end{equation}
Let
\begin{equation}
\left \{
\begin{array}
[c]{rl}%
\overline{X}_{k}\left(  t\right)  & =\left[
\begin{array}
[c]{c}%
x\left(  t\right) \\
x\left(  t-a\right) \\
\vdots \\
x\left(  t-\left(  k-1\right)  a\right)
\end{array}
\right]  \in \mathbf{R}^{kn},\\
\overline{U}_{k}\left(  t\right)  & =x\left(  t+b-ka\right)  \in \mathbf{R}%
^{n},\\
\overline{Y}_{k}\left(  t\right)  & =x\left(  t+b\right)  \in \mathbf{R}^{n}.
\end{array}
\right.  \label{eq73}%
\end{equation}
Then $\left(  \overline{U}_{k}\left(  t\right)  ,\overline{X}_{k}\left(
t\right)  ,\overline{Y}_{k}\left(  t\right)  \right)  $ satisfies%
\begin{equation}
\left.
\begin{array}
[c]{rl}%
\overline{X}_{k}\left(  t+b-a\right)  & =\overline{A}_{k}\overline{X}%
_{k}\left(  t\right)  +\overline{B}_{k}\overline{U}_{k}\left(  t\right)  ,\\
\overline{Y}_{k}\left(  t\right)  & =\overline{C}_{k}\overline{X}_{k}\left(
t\right)  +\overline{D}_{k}\overline{U}_{k}\left(  t\right)  .
\end{array}
\right \}  \label{eq74}%
\end{equation}

\end{lemma}

\begin{proof}
It is straightforward to see that, for any $i=0,1,\ldots,k-1,$%
\begin{align}
x\left(  t-ia\right)  =  &  Bx\left(  t-ia-b\right)  +Ax\left(  t-\left(
i+1\right)  a\right) \nonumber \\
=  &  Bx\left(  t-ia-b\right)  +A\left(  Bx\left(  t-\left(  i+1\right)
a-b\right)  +Ax\left(  t-\left(  i+2\right)  a\right)  \right) \nonumber \\
=  &  Bx\left(  t-ia-b\right)  +ABx\left(  t-\left(  i+1\right)  a-b\right)
+A^{2}x\left(  t-\left(  i+2\right)  a\right) \nonumber \\
=  &  \cdots \nonumber \\
=  &  Bx\left(  t-ia-b\right)  +ABx\left(  t-\left(  i+1\right)  a-b\right)
+\cdots \nonumber \\
&  +A^{k-i-1}Bx\left(  t-\left(  k-1\right)  a-b\right)  +A^{k-i}x\left(
t-ka\right)  . \label{eqk0}%
\end{align}
For $i=1,2,\ldots,k-1,$ we write the above $k-1$ equations in the dense form%
\[
\left[
\begin{array}
[c]{c}%
x\left(  t-a\right) \\
x\left(  t-2a\right) \\
\vdots \\
x\left(  t-\left(  k-1\right)  a\right) \\
x\left(  t-ka\right)
\end{array}
\right]  =\left[
\begin{array}
[c]{ccccc}%
0 & B & AB & \cdots & A^{k-2}B\\
& 0 & B & \cdots & A^{k-3}B\\
&  & \ddots & \ddots & \vdots \\
&  &  & 0 & B\\
&  &  &  & 0
\end{array}
\right]  \left[
\begin{array}
[c]{c}%
x\left(  t-b\right) \\
x\left(  t-a-b\right) \\
\vdots \\
x\left(  t-\left(  k-2\right)  a-b\right) \\
x\left(  t-\left(  k-1\right)  a-b\right)
\end{array}
\right]  +\left[
\begin{array}
[c]{c}%
A^{k-1}\\
A^{k-2}\\
\vdots \\
A\\
I_{n}%
\end{array}
\right]  x\left(  t-ka\right)  ,
\]
which can be written as%
\begin{equation}
\overline{X}_{k}\left(  t-a\right)  =\overline{A}_{k}\overline{X}_{k}\left(
t-b\right)  +\overline{B}_{k}x\left(  t-ka\right)  , \label{eq86}%
\end{equation}
which is just the first equation (\ref{eq74}). On the other hand, with $i=0$
in (\ref{eqk0}), we have%
\begin{align*}
x\left(  t\right)   &  =Bx\left(  t-b\right)  +ABx\left(  t-a-b\right)
+\cdots+A^{k-1}Bx\left(  t-\left(  k-1\right)  a-b\right)  +A^{k}x\left(
t-ka\right) \\
&  =\left[
\begin{array}
[c]{cccc}%
B & AB & \cdots & A^{k-1}B
\end{array}
\right]  X_{k}\left(  t-b\right)  +A^{k}x\left(  t-ka\right) \\
&  =\overline{C}_{k}X_{k}\left(  t-b\right)  +\overline{D}_{k}x\left(
t-ka\right)  ,
\end{align*}
which is just the second equation in (\ref{eq74}). The proof is finished.
\end{proof}

We then can present for Lemma \ref{th2} a time-domain interpretation, which
parallels Proposition \ref{pp1}.

\begin{proposition}
For any integer $k\geq1,$\ let $\overline{\mathit{\Omega}}_{k}$ be defined in
(\ref{eqomegak}). Consider the following LKF
\begin{equation}
\overline{V}_{k}\left(  x_{t}\right)  =\int_{t-a}^{t}\overline{X}%
_{k}^{\mathrm{T}}\left(  s\right)  \overline{Q}_{k}\overline{X}_{k}\left(
s\right)  \mathrm{d}s+\int_{t-b}^{t-a}\overline{X}_{k}^{\mathrm{T}}\left(
s\right)  \overline{P}_{k}\overline{X}_{k}\left(  s\right)  \mathrm{d}s,
\label{eqv}%
\end{equation}
where $\overline{P}_{k}=\overline{P}_{k}^{\mathrm{T}}\in \mathbf{R}^{kn\times
kn}$ and $\overline{Q}_{k}=\overline{Q}_{k}^{\mathrm{T}}\in \mathbf{R}%
^{kn\times kn}.$ Then%
\begin{equation}
\dot{\overline{V}}_{k}\left(  x_{t}\right)  =\left[
\begin{array}
[c]{c}%
\overline{X}_{k}\left(  t-b\right) \\
x\left(  t-ka\right)
\end{array}
\right]  ^{\mathrm{T}}\overline{\mathit{\Omega}}_{k}\left(  \overline{P}%
_{k},\overline{Q}_{k}\right)  \left[
\begin{array}
[c]{c}%
\overline{X}_{k}\left(  t-b\right) \\
x\left(  t-ka\right)
\end{array}
\right]  . \label{eqdv}%
\end{equation}

\end{proposition}

\begin{proof}
By using (\ref{eqk0}) and noting the structures of $\overline{\mathscr{A}}%
_{k}$ and $\overline{\mathscr{B}}_{k}$ in (\ref{eqab2}), we have%
\begin{align*}
\overline{X}_{k}\left(  t\right)   &  =\left[
\begin{array}
[c]{cccccc}%
B & AB & A^{2}B & \cdots & A^{k-1}B & A^{k}\\
& B & AB & \ddots & A^{k-2}B & A^{k-1}\\
&  & \ddots & \ddots & \vdots & \vdots \\
&  &  & B & AB & A^{2}\\
&  &  &  & B & A
\end{array}
\right]  \left[
\begin{array}
[c]{c}%
x\left(  t-b\right) \\
x\left(  t-a-b\right) \\
\vdots \\
x\left(  t-\left(  k-1\right)  a-b\right) \\
x\left(  t-ka\right)
\end{array}
\right] \\
&  =\left[
\begin{array}
[c]{cc}%
\overline{\mathscr{A}}_{k} & \overline{\mathscr{B}}_{k}%
\end{array}
\right]  \left[
\begin{array}
[c]{c}%
\overline{X}_{k}\left(  t-b\right) \\
x\left(  t-ka\right)
\end{array}
\right]  .
\end{align*}
Therefore, it follows from (\ref{eq86}) that
\begin{align*}
\dot{\overline{V}}_{k}\left(  x_{t}\right)  =  &  \overline{X}_{k}%
^{\mathrm{T}}\left(  t-a\right)  \overline{P}_{k}\overline{X}_{k}\left(
t-a\right)  -\overline{X}_{k}^{\mathrm{T}}\left(  t-b\right)  \overline{P}%
_{k}\overline{X}_{k}\left(  t-b\right) \\
&  +\overline{X}_{k}^{\mathrm{T}}\left(  t\right)  \overline{Q}_{k}%
\overline{X}_{k}\left(  t\right)  -\overline{X}_{k}^{\mathrm{T}}\left(
t-a\right)  \overline{Q}_{k}\overline{X}_{k}\left(  t-a\right) \\
=  &  \left(  \overline{A}_{k}\overline{X}_{k}\left(  t-b\right)
+\overline{B}_{k}x\left(  t-ka\right)  \right)  ^{\mathrm{T}}\left(
\overline{P}_{k}-\overline{Q}_{k}\right)  \left(  \overline{A}_{k}\overline
{X}_{k}\left(  t-b\right)  +\overline{B}_{k}x\left(  t-ka\right)  \right) \\
&  +\overline{X}_{k}^{\mathrm{T}}\left(  t\right)  \overline{Q}_{k}%
\overline{X}_{k}\left(  t\right)  -\overline{X}_{k}^{\mathrm{T}}\left(
t-b\right)  \overline{P}_{k}\overline{X}_{k}\left(  t-b\right) \\
=  &  \left[
\begin{array}
[c]{c}%
\overline{X}_{k}\left(  t-b\right) \\
x\left(  t-ka\right)
\end{array}
\right]  ^{\mathrm{T}}\left[
\begin{array}
[c]{cc}%
\overline{A}_{k} & \overline{B}_{k}%
\end{array}
\right]  ^{\mathrm{T}}\left(  \overline{P}_{k}-\overline{Q}_{k}\right)
\left[
\begin{array}
[c]{cc}%
\overline{A}_{k} & \overline{B}_{k}%
\end{array}
\right]  \left[
\begin{array}
[c]{c}%
\overline{X}_{k}\left(  t-b\right) \\
x\left(  t-ka\right)
\end{array}
\right] \\
&  -\left[
\begin{array}
[c]{c}%
\overline{X}_{k}\left(  t-b\right) \\
x\left(  t-ka\right)
\end{array}
\right]  ^{\mathrm{T}}\left[
\begin{array}
[c]{cc}%
I_{kn} & 0_{kn\times n}%
\end{array}
\right]  ^{\mathrm{T}}\overline{P}_{k}\left[
\begin{array}
[c]{cc}%
I_{kn} & 0_{kn\times n}%
\end{array}
\right]  \left[
\begin{array}
[c]{c}%
\overline{X}_{k}\left(  t-b\right) \\
x\left(  t-ka\right)
\end{array}
\right] \\
&  +\left[
\begin{array}
[c]{c}%
\overline{X}_{k}\left(  t-b\right) \\
x\left(  t-ka\right)
\end{array}
\right]  ^{\mathrm{T}}\left[
\begin{array}
[c]{cc}%
\overline{\mathscr{A}}_{k} & \overline{\mathscr{B}}_{k}%
\end{array}
\right]  ^{\mathrm{T}}\overline{Q}_{k}\left[
\begin{array}
[c]{cc}%
\overline{\mathscr{A}}_{k} & \overline{\mathscr{B}}_{k}%
\end{array}
\right]  \left[
\begin{array}
[c]{c}%
\overline{X}_{k}\left(  t-b\right) \\
x\left(  t-ka\right)
\end{array}
\right] \\
=  &  \left[
\begin{array}
[c]{c}%
\overline{X}_{k}\left(  t-b\right) \\
x\left(  t-ka\right)
\end{array}
\right]  ^{\mathrm{T}}\overline{\mathit{\Omega}}_{k}\left(  \overline{P}%
_{k},\overline{Q}_{k}\right)  \left[
\begin{array}
[c]{c}%
\overline{X}_{k}\left(  t-b\right) \\
x\left(  t-ka\right)
\end{array}
\right]  ,
\end{align*}
which completes the proof.
\end{proof}

One may wonder the relationship between Theorem \ref{th4} and Lemma \ref{th2}.
Such a relationship should be revealed from the time-domain interpretations of
these two LMIs. To investigate this problem, we need to find the relationship
between $\overline{\mathit{\Omega}}_{k}$ and $\mathit{\Omega}_{k}.$ Such a
relationship should be revealed from the time-domain interpretations of these
two LMIs, say, the relationship between $X_{k}\left(  t\right)  $ and
$\overline{X}_{k}\left(  t\right)  ,$ and the relationship between
\[
\left[
\begin{array}
[c]{c}%
X_{k}\left(  t-b\right)  \\
x\left(  t-b-ka\right)
\end{array}
\right]  \text{ and }\left[
\begin{array}
[c]{c}%
\overline{X}_{k}\left(  t-b\right)  \\
x\left(  t-ka\right)
\end{array}
\right]  .
\]
To this end, we denote, for any integer $k\geq1,$
\[
W_{k}=\left[
\begin{array}
[c]{cccc}%
B^{\left[  k-1\right]  } & B^{\left[  k-2\right]  }A^{\left[  1\right]  } &
\cdots & A^{\left[  k-1\right]  }\\
& \ddots & \ddots & \vdots \\
&  & B^{\left[  1\right]  } & A^{\left[  1\right]  }\\
&  &  & I_{n}%
\end{array}
\right]  ,T_{k}=\left[
\begin{array}
[c]{cc}%
W_{k} & \\
& I_{n}%
\end{array}
\right]  .
\]
Then we have the following result.

\begin{proposition}
Let $\mathit{\Omega}_{k}\left(  P_{k},Q_{k}\right)  $ and $\overline
{\mathit{\Omega}}_{k}\left(  \overline{P}_{k},\overline{Q}_{k}\right)  $ be
defined respectively in (\ref{eq77b}) and (\ref{eqomegak}). Let%
\begin{equation}
P_{k}=W_{k}^{\mathrm{T}}\overline{P}_{k}W_{k},\;Q_{k}=W_{k}^{\mathrm{T}%
}\overline{Q}_{k}W_{k}. \label{eq45}%
\end{equation}
Then there holds%
\begin{equation}
\mathit{\Omega}_{k}\left(  P_{k},Q_{k}\right)  =T_{k}^{\mathrm{T}}%
\overline{\mathit{\Omega}}_{k}\left(  \overline{P}_{k},\overline{Q}%
_{k}\right)  T_{k}. \label{eq46}%
\end{equation}
Therefore, the LMI in (\ref{eqz}) is feasible if and only if the LMI in
(\ref{eq55}) is feasible.
\end{proposition}

\begin{proof}
By using Lemma \ref{lm10} we have%
\begin{align*}
\overline{X}_{k}\left(  t-b\right)   &  =\left[
\begin{array}
[c]{c}%
x\left(  t-b\right)  \\
x\left(  t-a-b\right)  \\
\vdots \\
x\left(  t-\left(  k-1\right)  a-b\right)
\end{array}
\right]  \\
&  =\left[
\begin{array}
[c]{cccc}%
B^{\left[  k-1\right]  } & B^{\left[  k-2\right]  }A^{\left[  1\right]  } &
\cdots & A^{\left[  k-1\right]  }\\
& \ddots & \ddots & \vdots \\
&  & B^{\left[  1\right]  } & A^{\left[  1\right]  }\\
&  &  & I_{n}%
\end{array}
\right]  \left[
\begin{array}
[c]{c}%
x\left(  t-kb\right)  \\
x\left(  t-\left(  k-1\right)  b-a\right)  \\
\vdots \\
x\left(  t-2b-\left(  k-2\right)  a\right)  \\
x\left(  t-b-\left(  k-1\right)  a\right)
\end{array}
\right]  \\
&  =W_{k}X_{k}\left(  t\right)  ,
\end{align*}
from which we get%
\begin{align*}
\left[
\begin{array}
[c]{c}%
\overline{X}_{k}\left(  t-2b\right)  \\
x\left(  t-b-ka\right)
\end{array}
\right]   &  =\left[
\begin{array}
[c]{c}%
W_{k}X_{k}\left(  t-b\right)  \\
x\left(  t-b-ka\right)
\end{array}
\right]  \\
&  =\left[
\begin{array}
[c]{cc}%
W_{k} & 0\\
0 & I_{n}%
\end{array}
\right]  \left[
\begin{array}
[c]{c}%
X_{k}\left(  t-b\right)  \\
x\left(  t-b-ka\right)
\end{array}
\right]  \\
&  =T_{k}\left[
\begin{array}
[c]{c}%
X_{k}\left(  t-b\right)  \\
x\left(  t-b-ka\right)
\end{array}
\right]  .
\end{align*}
Therefore, we have from (\ref{eqv}) that%
\begin{align}
\overline{V}_{k}\left(  x_{t-b}\right)   &  =\int_{t-a}^{t}\overline{X}%
_{k}^{\mathrm{T}}\left(  s-b\right)  \overline{Q}_{k}\overline{X}_{k}\left(
s-b\right)  \mathrm{d}s+\int_{t-b}^{t-a}\overline{X}_{k}^{\mathrm{T}}\left(
s-b\right)  \overline{P}_{k}\overline{X}_{k}\left(  s-b\right)  \mathrm{d}%
s,\nonumber \\
&  =\int_{t-a}^{t}X_{k}^{\mathrm{T}}\left(  s\right)  W_{k}^{\mathrm{T}%
}\overline{Q}_{k}W_{k}X_{k}\left(  s\right)  \mathrm{d}s+\int_{t-b}^{t-a}%
X_{k}^{\mathrm{T}}\left(  s\right)  W_{k}^{\mathrm{T}}\overline{P}_{k}%
W_{k}X_{k}\left(  s\right)  \mathrm{d}s,\label{eq42}%
\end{align}
and from (\ref{eqdv}) that%
\begin{align}
\dot{\overline{V}}_{k}\left(  x_{t-b}\right)   &  =\left[
\begin{array}
[c]{c}%
\overline{X}_{k}\left(  t-2b\right)  \\
x\left(  t-b-ka\right)
\end{array}
\right]  ^{\mathrm{T}}\overline{\mathit{\Omega}}_{k}\left(  \overline{P}%
_{k},\overline{Q}_{k}\right)  \left[
\begin{array}
[c]{c}%
\overline{X}_{k}\left(  t-2b\right)  \\
x\left(  t-b-ka\right)
\end{array}
\right]  \nonumber \\
&  =\left[
\begin{array}
[c]{c}%
X_{k}\left(  t-b\right)  \\
x\left(  t-b-ka\right)
\end{array}
\right]  ^{\mathrm{T}}T_{k}^{\mathrm{T}}\overline{\mathit{\Omega}}_{k}\left(
\overline{P}_{k},\overline{Q}_{k}\right)  T_{k}\left[
\begin{array}
[c]{c}%
X_{k}\left(  t-b\right)  \\
x\left(  t-b-ka\right)
\end{array}
\right]  .\label{eq43}%
\end{align}
By comparing (\ref{eq42}) and (\ref{eq43}) with (\ref{eqv1}) and (\ref{eqdv1})
we know that, if $\left(  P_{k},Q_{k}\right)  $ satisfies (\ref{eq45}), then
$\mathit{\Omega}_{k}$ and $\overline{\mathit{\Omega}}_{k}$ satisfies
(\ref{eq46}). The proof is finished.
\end{proof}

It follows that Theorem \ref{th4} is equivalent to Lemma \ref{th2}. Even so,
Theorem \ref{th4} possesses great advantage over Lemma \ref{th2} since the
system parameters appear linearly (quadratically) in the LMIs (\ref{eq55}),
which has been very important in the robust stability analysis. We next show
the connection to the Carvalho Condition.

\begin{lemma}
\label{lm3} \cite{carvalho96laa} The linear difference equation (\ref{sys}) is
exponentially stable if there exist two positive definite matrices $X_{1}%
\in \mathbf{R}^{n\times n}$ and $Y_{1}\in \mathbf{R}^{n\times n}$ such that the
following LMI is satisfied%
\begin{equation}
\mathit{\Phi}_{1}\left(  X_{1},Y_{1}\right)  =\left[
\begin{array}
[c]{cc}%
A & B\\
I_{n} & 0
\end{array}
\right]  ^{\mathrm{T}}\left[
\begin{array}
[c]{cc}%
X_{1} & 0\\
0 & Y_{1}%
\end{array}
\right]  \left[
\begin{array}
[c]{cc}%
A & B\\
I_{n} & 0
\end{array}
\right]  -\left[
\begin{array}
[c]{cc}%
X_{1} & 0\\
0 & Y_{1}%
\end{array}
\right]  <0. \label{eq6}%
\end{equation}

\end{lemma}

\begin{proof}
For future use, we give a simple proof here. Choose the following LK
functional%
\begin{equation}
W_{1}\left(  x_{t}\right)  =\int_{t-a}^{t}x^{\mathrm{T}}\left(  s\right)
X_{1}x\left(  s\right)  \mathrm{d}s+\int_{t-b}^{t-a}x^{\mathrm{T}}\left(
s\right)  Y_{1}x\left(  s\right)  \mathrm{d}s, \label{eq53}%
\end{equation}
which is such that%
\begin{align}
\dot{W}_{1}\left(  x_{t}\right)   &  =x^{\mathrm{T}}\left(  t\right)
X_{1}x\left(  t\right)  -x^{\mathrm{T}}\left(  t-a\right)  X_{1}x\left(
t-a\right)  +x^{\mathrm{T}}\left(  t-a\right)  Y_{1}x\left(  t-a\right)
-x^{\mathrm{T}}\left(  t-b\right)  Y_{1}x\left(  t-b\right) \nonumber \\
&  =\left[
\begin{array}
[c]{c}%
x\left(  t-a\right) \\
x\left(  t-b\right)
\end{array}
\right]  ^{\mathrm{T}}\mathit{\Phi}_{1}\left(  X_{1},Y_{1}\right)  \left[
\begin{array}
[c]{c}%
x\left(  t-a\right) \\
x\left(  t-b\right)
\end{array}
\right]  . \label{eq56}%
\end{align}
Since $\mathit{\Phi}_{1}\left(  X_{1},Y_{1}\right)  <0,$ the stability follows
from the Lyapunov stability theorem \cite{carvalho96laa}.
\end{proof}

If we set $k=1$ in Theorem \ref{th4} and denote%
\[
E_{2}=\left[
\begin{array}
[c]{cc}%
0 & I_{n}\\
I_{n} & 0
\end{array}
\right]  ,E_{3}=\left[
\begin{array}
[c]{ccc}%
0 & 0 & I_{n}\\
0 & I_{n} & 0\\
I_{n} & 0 & 0
\end{array}
\right]  ,
\]
we obtain the following result.

\begin{lemma}
\label{lm3a}Let $\mathit{\Omega}_{k}$ be defined in (\ref{eq77b}),
$\overline{\mathit{\Omega}}_{k}$ be defined in (\ref{eqomegak}) and
$\mathit{\Phi}_{1}$ be defined in (\ref{eq6}). Then, for $k=1$, there holds%
\begin{align}
\mathit{\Omega}_{1}\left(  P_{1},Q_{1}\right)   &  =E_{2}^{\mathrm{T}%
}\mathit{\Phi}_{1}\left(  Q_{1},P_{1}\right)  E_{2},\label{eq99a}\\
\overline{\mathit{\Omega}}_{1}\left(  \overline{P}_{1},\overline{Q}%
_{1}\right)   &  =E_{2}^{\mathrm{T}}\mathit{\Phi}_{1}\left(  \overline{Q}%
_{1},\overline{P}_{1}\right)  E_{2}. \label{eq99b}%
\end{align}
Thus the result in Lemma \ref{lm3} \cite{carvalho96laa} is a special case of
Lemma \ref{th2} and Theorem \ref{th4}.
\end{lemma}

\begin{proof}
Let $k=1.$ Then it follows from (\ref{eqv1}) that%
\begin{align}
V_{1}\left(  x_{t+b}\right)   &  =\int_{t-b}^{t-a}X_{1}^{\mathrm{T}}\left(
s+b\right)  P_{1}X_{1}\left(  s+b\right)  \mathrm{d}s+\int_{t-a}^{t}%
X_{1}^{\mathrm{T}}\left(  s+b\right)  Q_{1}X_{1}\left(  s+b\right)
\mathrm{d}s\nonumber \\
&  =\int_{t-b}^{t-a}x^{\mathrm{T}}\left(  s\right)  P_{1}x\left(  s\right)
\mathrm{d}s+\int_{t-a}^{t}x^{\mathrm{T}}\left(  s\right)  Q_{1}x\left(
s\right)  \mathrm{d}s, \label{eq50}%
\end{align}
and from (\ref{eqdv1}) that%
\begin{align}
\dot{V}_{1}\left(  x_{t+b}\right)   &  =\left[
\begin{array}
[c]{c}%
x\left(  t-b\right) \\
x\left(  t-a\right)
\end{array}
\right]  ^{\mathrm{T}}\mathit{\Omega}_{1}\left(  P_{1},Q_{1}\right)  \left[
\begin{array}
[c]{c}%
x\left(  t-b\right) \\
x\left(  t-a\right)
\end{array}
\right] \nonumber \\
&  =\left[
\begin{array}
[c]{c}%
x\left(  t-a\right) \\
x\left(  t-b\right)
\end{array}
\right]  ^{\mathrm{T}}E_{2}^{\mathrm{T}}\mathit{\Omega}_{1}\left(  P_{1}%
,Q_{1}\right)  E_{2}\left[
\begin{array}
[c]{c}%
x\left(  t-a\right) \\
x\left(  t-b\right)
\end{array}
\right]  . \label{eq51}%
\end{align}
By comparing (\ref{eq50}) and (\ref{eq51}) with (\ref{eq53}) and (\ref{eq56}),
respectively, we get (\ref{eq99a}).

Similarly, we have from (\ref{eqv}) that%
\begin{align}
\overline{V}_{1}\left(  x_{t}\right)   &  =\int_{t-a}^{t}\overline{X}%
_{1}^{\mathrm{T}}\left(  s\right)  \overline{Q}_{1}\overline{X}_{1}\left(
s\right)  \mathrm{d}s+\int_{t-b}^{t-a}\overline{X}_{1}^{\mathrm{T}}\left(
s\right)  \overline{P}_{1}\overline{X}_{1}\left(  s\right)  \mathrm{d}%
s\nonumber \\
&  =\int_{t-a}^{t}x^{\mathrm{T}}\left(  s\right)  \overline{Q}_{1}x\left(
s\right)  \mathrm{d}s+\int_{t-b}^{t-a}x^{\mathrm{T}}\left(  s\right)
\overline{P}_{1}x\left(  s\right)  \mathrm{d}s, \label{eq52}%
\end{align}
and from (\ref{eqdv}) that%
\begin{align}
\dot{\overline{V}}_{1}\left(  x_{t}\right)   &  =\left[
\begin{array}
[c]{c}%
x\left(  t-b\right) \\
x\left(  t-a\right)
\end{array}
\right]  ^{\mathrm{T}}\overline{\mathit{\Omega}}_{1}\left(  \overline{P}%
_{1},\overline{Q}_{1}\right)  \left[
\begin{array}
[c]{c}%
x\left(  t-b\right) \\
x\left(  t-a\right)
\end{array}
\right] \nonumber \\
&  =\left[
\begin{array}
[c]{c}%
x\left(  t-a\right) \\
x\left(  t-b\right)
\end{array}
\right]  ^{\mathrm{T}}E_{2}^{\mathrm{T}}\overline{\mathit{\Omega}}_{1}\left(
\overline{P}_{1},\overline{Q}_{1}\right)  E_{2}\left[
\begin{array}
[c]{c}%
x\left(  t-a\right) \\
x\left(  t-b\right)
\end{array}
\right]  . \label{eq57}%
\end{align}
By comparing (\ref{eq52}) and (\ref{eq57}) with (\ref{eq53}) and (\ref{eq56}),
respectively, we get (\ref{eq99a}). The proof is finished.
\end{proof}

We next investigate the relationship between Theorem \ref{th4} and a result in
\cite{ddmb16tac}. To this end, we denote
\begin{align*}
N_{21}  &  =\left[
\begin{array}
[c]{ccc}%
A & B & 0\\
I_{n} & 0 & 0
\end{array}
\right]  ,\;N_{22}=\left[
\begin{array}
[c]{ccc}%
0 & 0 & I_{n}\\
0 & I_{n} & 0
\end{array}
\right]  ,\\
M_{21}  &  =\left[
\begin{array}
[c]{ccc}%
A & B & 0\\
0 & 0 & I_{n}%
\end{array}
\right]  ,\;M_{22}=\left[
\begin{array}
[c]{ccc}%
I_{n} & 0 & 0\\
0 & I_{n} & 0
\end{array}
\right]  .
\end{align*}

\begin{lemma}
\label{lm4} \cite{ddmb16tac} The linear difference equation (\ref{sys}) is
exponentially stable if there exist four positive definite matrices
$X_{2},Y_{2}\in \mathbf{R}^{2n\times2n},U_{2},V_{2}\in \mathbf{R}^{n\times n},$
such that the following LMI is satisfied%
\begin{equation}
\mathit{\Phi}_{2}\left(  X_{2}^{\ast},Y_{2}^{\ast}\right)  =N_{21}%
^{\mathrm{T}}X_{2}^{\ast}N_{21}-N_{22}^{\mathrm{T}}X_{2}^{\ast}N_{22}%
+M_{21}^{\mathrm{T}}Y_{2}^{\ast}M_{21}-M_{22}^{\mathrm{T}}Y_{2}^{\ast}%
M_{22}<0, \label{eq7a}%
\end{equation}
where%
\begin{align}
Y_{2}^{\ast}  &  =Y_{2}+\left[
\begin{array}
[c]{cc}%
U_{2}+V_{2} & 0\\
0 & 0_{n\times n}%
\end{array}
\right]  >0,\label{eq40}\\
X_{2}^{\ast}  &  =X_{2}+\left[
\begin{array}
[c]{cc}%
0_{n\times n} & 0\\
0 & V_{2}%
\end{array}
\right]  >0. \label{eq41}%
\end{align}

\end{lemma}

\begin{proof}
This lemma is a little different from the original result in \cite{ddmb16tac}
and thus a simple proof will be provided for completeness (also for the
purpose of further using). Choose a more general LKF candidate as
\cite{ddmb16tac} (where we have assumed without loss of generality that
$\mu=0$)
\begin{align}
W_{2}\left(  x_{t}\right)  =  &  \int_{t-a}^{t}x^{\mathrm{T}}\left(  s\right)
U_{2}x\left(  s\right)  \mathrm{d}s+\int_{t-b}^{t}x^{\mathrm{T}}\left(
s\right)  V_{2}x\left(  s\right)  \mathrm{d}s,\nonumber \\
&  +\int_{t-c}^{t}\left[
\begin{array}
[c]{c}%
x\left(  s\right) \\
x\left(  s-a\right)
\end{array}
\right]  ^{\mathrm{T}}X_{2}\left[
\begin{array}
[c]{c}%
x\left(  s\right) \\
x\left(  s-a\right)
\end{array}
\right]  \mathrm{d}s,\nonumber \\
&  +\int_{t-b}^{t-c}\left[
\begin{array}
[c]{c}%
x\left(  s+c\right) \\
x\left(  s\right)
\end{array}
\right]  ^{\mathrm{T}}Y_{2}\left[
\begin{array}
[c]{c}%
x\left(  s+c\right) \\
x\left(  s\right)
\end{array}
\right]  \mathrm{d}s,\nonumber
\end{align}
where $c=b-a$. It can be verified that%
\begin{align*}
&  \int_{t-a}^{t}x^{\mathrm{T}}\left(  s\right)  U_{2}x\left(  s\right)
\mathrm{d}s+\int_{t-b}^{t}x^{\mathrm{T}}\left(  s\right)  V_{2}x\left(
s\right)  \mathrm{d}s\\
&  =\int_{t-a}^{t}x^{\mathrm{T}}\left(  s\right)  \left(  U_{2}+V_{2}\right)
x\left(  s\right)  \mathrm{d}s+\int_{t-b}^{t-a}x^{\mathrm{T}}\left(  s\right)
V_{2}x\left(  s\right)  \mathrm{d}s\\
&  =\int_{t-b}^{t-c}x^{\mathrm{T}}\left(  s+c\right)  \left(  U_{2}%
+V_{2}\right)  x\left(  s+c\right)  \mathrm{d}s+\int_{t-c}^{t}x^{\mathrm{T}%
}\left(  s-a\right)  V_{2}x\left(  s-a\right)  \mathrm{d}s,
\end{align*}
from which it follows that%
\begin{align}
W_{2}\left(  x_{t}\right)  =  &  \int_{t-c}^{t}\left[
\begin{array}
[c]{c}%
x\left(  s\right) \\
x\left(  s-a\right)
\end{array}
\right]  ^{\mathrm{T}}X_{2}^{\ast}\left[
\begin{array}
[c]{c}%
x\left(  s\right) \\
x\left(  s-a\right)
\end{array}
\right]  \mathrm{d}s\nonumber \\
&  +\int_{t-b}^{t-c}\left[
\begin{array}
[c]{c}%
x\left(  s+c\right) \\
x\left(  s\right)
\end{array}
\right]  ^{\mathrm{T}}Y_{2}^{\ast}\left[
\begin{array}
[c]{c}%
x\left(  s+c\right) \\
x\left(  s\right)
\end{array}
\right]  \mathrm{d}s, \label{eqw2}%
\end{align}
whose time-derivative can be evaluated as%
\begin{align}
\dot{W}_{2}\left(  x_{t}\right)  =  &  \left[
\begin{array}
[c]{c}%
x\left(  t\right) \\
x\left(  t-a\right)
\end{array}
\right]  ^{\mathrm{T}}X_{2}^{\ast}\left[
\begin{array}
[c]{c}%
x\left(  t\right) \\
x\left(  t-a\right)
\end{array}
\right]  -\left[
\begin{array}
[c]{c}%
x\left(  t-c\right) \\
x\left(  t-b\right)
\end{array}
\right]  ^{\mathrm{T}}X_{2}^{\ast}\left[
\begin{array}
[c]{c}%
x\left(  t-c\right) \\
x\left(  t-b\right)
\end{array}
\right] \nonumber \\
&  +\left[
\begin{array}
[c]{c}%
x\left(  t\right) \\
x\left(  t-c\right)
\end{array}
\right]  ^{\mathrm{T}}Y_{2}^{\ast}\left[
\begin{array}
[c]{c}%
x\left(  t\right) \\
x\left(  t-c\right)
\end{array}
\right]  -\left[
\begin{array}
[c]{c}%
x\left(  t-a\right) \\
x\left(  t-b\right)
\end{array}
\right]  ^{\mathrm{T}}Y_{2}^{\ast}\left[
\begin{array}
[c]{c}%
x\left(  t-a\right) \\
x\left(  t-b\right)
\end{array}
\right] \nonumber \\
=  &  \xi_{2}^{\mathrm{T}}\left(  t\right)  \mathit{\Phi}_{2}\xi_{2}\left(
t\right)  , \label{eqdw2}%
\end{align}
where $\xi_{2}\left(  t\right)  =[x^{\mathrm{T}}\left(  t-a\right)
,x^{\mathrm{T}}\left(  t-b\right)  ,x^{\mathrm{T}}\left(  t-c\right)
]^{\mathrm{T}}$. The result then follows again from the Lyapunov stability
theorem \cite{carvalho96laa}.
\end{proof}

The decision matrices $U_{2}$ and $V_{2}$ in (\ref{eq7a}) are in fact
redundant, as shown in the following corollary.

\begin{corollary}
\label{coro1}There exist four positive definite matrices $X_{2},Y_{2}%
\in \mathbf{R}^{2n\times2n},U_{2},V_{2}\in \mathbf{R}^{n\times n}$ such that
(\ref{eq7a}) is satisfied if and only if there exist two positive definite
matrices $X_{2}^{\ast},Y_{2}^{\ast}\in \mathbf{R}^{2n\times2n}$ such that
(\ref{eq7a}) is satisfied.
\end{corollary}

\begin{proof}
If $X_{2}>0,Y_{2}>0,U_{2}>0,V_{2}>0,$ then it follows from (\ref{eq40}%
)-(\ref{eq41}) that $X_{2}^{\ast}>0,Y_{2}^{\ast}>0.$ On the other hand, if
$X_{2}^{\ast}>0,Y_{2}^{\ast}>0$, we can always find $X_{2}>0,Y_{2}%
>0,U_{2}>0,V_{2}>0,$ satisfying (\ref{eq40})-(\ref{eq41}), for example,
$U_{2}=V_{2}=\varepsilon I_{n},$ where $\varepsilon>0$ is sufficiently small.
The proof is finished.
\end{proof}

We then can state the following result which connects the result in this paper
and the one in \cite{ddmb16tac}.

\begin{proposition}
Let $\left(  P_{2},Q_{2}\right)  ,\left(  \overline{P}_{2},\overline{Q}%
_{2}\right)  $ and $\left(  X_{2}^{\ast},Y_{2}^{\ast}\right)  $ be related
with
\begin{align}
P_{2}  &  =W_{2}^{\mathrm{T}}X_{2}^{\ast}W_{2},\;Q_{2}=E_{2}^{\mathrm{T}}%
Y_{2}^{\ast}E_{2},\label{eqpq2}\\
X_{2}^{\ast}  &  =\overline{P}_{2},\;Y_{2}^{\ast}=E_{2}^{\mathrm{T}}%
W_{2}^{\mathrm{T}}\overline{Q}_{2}W_{2}E_{2}. \label{eqpq2a}%
\end{align}
Then $\mathit{\Omega}_{2}\left(  P_{2},Q_{2}\right)  $ and $\mathit{\Phi}%
_{2}\left(  X_{2}^{\ast},Y_{2}^{\ast}\right)  $ satisfy
\begin{align}
\mathit{\Omega}_{2}\left(  P_{2},Q_{2}\right)   &  =T_{3}^{\mathrm{T}}%
E_{3}^{\mathrm{T}}\mathit{\Phi}_{2}\left(  X_{2}^{\ast},Y_{2}^{\ast}\right)
E_{3}T_{3},\label{eqpq3}\\
\mathit{\Phi}_{2}\left(  X_{2}^{\ast},Y_{2}^{\ast}\right)   &  =E_{3}%
^{\mathrm{T}}\overline{\mathit{\Omega}}_{2}\left(  \overline{P}_{2}%
,\overline{Q}_{2}\right)  E_{3}. \label{eqpq3a}%
\end{align}
Thus the result in Lemma \ref{lm4} \cite{ddmb16tac} is a special case of Lemma
\ref{th2} and Theorem \ref{th4}.
\end{proposition}

\begin{proof}
Notice from (\ref{eqv1}) that%
\begin{align}
V_{2}\left(  x_{t+b}\right)   &  =\int_{t-b}^{t-a}X_{2}^{\mathrm{T}}\left(
s+b\right)  P_{2}X_{2}\left(  s+b\right)  \mathrm{d}s+\int_{t-a}^{t}%
X_{2}^{\mathrm{T}}\left(  s+b\right)  Q_{2}X_{2}\left(  s+b\right)
\mathrm{d}s\nonumber \\
&  =\int_{t-b}^{t-a}\left[
\begin{array}
[c]{c}%
x\left(  s-b\right) \\
x\left(  s-a\right)
\end{array}
\right]  ^{\mathrm{T}}P_{2}\left[
\begin{array}
[c]{c}%
x\left(  s-b\right) \\
x\left(  s-a\right)
\end{array}
\right]  \mathrm{d}s+\int_{t-a}^{t}\left[
\begin{array}
[c]{c}%
x\left(  s-b\right) \\
x\left(  s-a\right)
\end{array}
\right]  ^{\mathrm{T}}Q_{2}\left[
\begin{array}
[c]{c}%
x\left(  s-b\right) \\
x\left(  s-a\right)
\end{array}
\right]  \mathrm{d}s, \label{eq60}%
\end{align}
and from (\ref{eqdv1}) that%
\begin{align}
\dot{V}_{2}\left(  x_{t+b}\right)   &  =\left[
\begin{array}
[c]{c}%
X_{2}\left(  t\right) \\
x\left(  t-2a\right)
\end{array}
\right]  ^{\mathrm{T}}\mathit{\Omega}_{2}\left(  P_{2},Q_{2}\right)  \left[
\begin{array}
[c]{c}%
X_{2}\left(  t\right) \\
x\left(  t-2a\right)
\end{array}
\right] \nonumber \\
&  =\left[
\begin{array}
[c]{c}%
x\left(  t-2b\right) \\
x\left(  t-a-b\right) \\
x\left(  t-2a\right)
\end{array}
\right]  ^{\mathrm{T}}\mathit{\Omega}_{2}\left(  P_{2},Q_{2}\right)  \left[
\begin{array}
[c]{c}%
x\left(  t-2b\right) \\
x\left(  t-a-b\right) \\
x\left(  t-2a\right)
\end{array}
\right]  . \label{eq61}%
\end{align}
On the other hand, we get from (\ref{eqw2}) that%
\begin{align*}
W_{2}\left(  x_{t}\right)   &  =\int_{t-\left(  b-a\right)  }^{t}\left[
\begin{array}
[c]{c}%
x\left(  s\right) \\
x\left(  s-a\right)
\end{array}
\right]  ^{\mathrm{T}}X_{2}^{\ast}\left[
\begin{array}
[c]{c}%
x\left(  s\right) \\
x\left(  s-a\right)
\end{array}
\right]  \mathrm{d}s+\int_{t-b}^{t-\left(  b-a\right)  }\left[
\begin{array}
[c]{c}%
x\left(  s+b-a\right) \\
x\left(  s\right)
\end{array}
\right]  ^{\mathrm{T}}Y_{2}^{\ast}\left[
\begin{array}
[c]{c}%
x\left(  s+b-a\right) \\
x\left(  s\right)
\end{array}
\right]  \mathrm{d}s\\
&  =\int_{t-b}^{t-a}\left[
\begin{array}
[c]{c}%
x\left(  s+a\right) \\
x\left(  s\right)
\end{array}
\right]  ^{\mathrm{T}}X_{2}^{\ast}\left[
\begin{array}
[c]{c}%
x\left(  s+a\right) \\
x\left(  s\right)
\end{array}
\right]  \mathrm{d}s+\int_{t-a}^{t}\left[
\begin{array}
[c]{c}%
x\left(  s\right) \\
x\left(  s-b+a\right)
\end{array}
\right]  ^{\mathrm{T}}Y_{2}^{\ast}\left[
\begin{array}
[c]{c}%
x\left(  s\right) \\
x\left(  s-b+a\right)
\end{array}
\right]  \mathrm{d}s,
\end{align*}
from which we have%
\begin{align}
W_{2}\left(  x_{t-a}\right)   &  =\int_{t-b}^{t-a}\left[
\begin{array}
[c]{c}%
x\left(  s\right) \\
x\left(  s-a\right)
\end{array}
\right]  ^{\mathrm{T}}X_{2}^{\ast}\left[
\begin{array}
[c]{c}%
x\left(  s\right) \\
x\left(  s-a\right)
\end{array}
\right]  \mathrm{d}s+\int_{t-a}^{t}\left[
\begin{array}
[c]{c}%
x\left(  s-a\right) \\
x\left(  s-b\right)
\end{array}
\right]  ^{\mathrm{T}}Y_{2}^{\ast}\left[
\begin{array}
[c]{c}%
x\left(  s-a\right) \\
x\left(  s-b\right)
\end{array}
\right]  \mathrm{d}s\nonumber \\
&  =\int_{t-b}^{t-a}\left[
\begin{array}
[c]{c}%
x\left(  s-b\right) \\
x\left(  s-a\right)
\end{array}
\right]  ^{\mathrm{T}}W_{2}^{\mathrm{T}}X_{2}^{\ast}W_{2}\left[
\begin{array}
[c]{c}%
x\left(  s-b\right) \\
x\left(  s-a\right)
\end{array}
\right]  \mathrm{d}s+\int_{t-a}^{t}\left[
\begin{array}
[c]{c}%
x\left(  s-b\right) \\
x\left(  s-a\right)
\end{array}
\right]  ^{\mathrm{T}}E_{2}^{\mathrm{T}}Y_{2}^{\ast}E_{2}\left[
\begin{array}
[c]{c}%
x\left(  s-b\right) \\
x\left(  s-a\right)
\end{array}
\right]  \mathrm{d}s. \label{eq62}%
\end{align}
Moreover, from (\ref{eqdw2}) we obtain%
\begin{align}
\dot{W}_{2}\left(  x_{t-a}\right)   &  =\left[
\begin{array}
[c]{c}%
x\left(  t-2a\right) \\
x\left(  t-a-b\right) \\
x\left(  t-b\right)
\end{array}
\right]  ^{\mathrm{T}}\mathit{\Phi}_{2}\left(  X_{2}^{\ast},Y_{2}^{\ast
}\right)  \left[
\begin{array}
[c]{c}%
x\left(  t-2a\right) \\
x\left(  t-a-b\right) \\
x\left(  t-b\right)
\end{array}
\right] \nonumber \\
&  =\left[
\begin{array}
[c]{c}%
x\left(  t-2b\right) \\
x\left(  t-a-b\right) \\
x\left(  t-2a\right)
\end{array}
\right]  ^{\mathrm{T}}T_{3}^{\mathrm{T}}E_{3}^{\mathrm{T}}\mathit{\Phi}%
_{2}\left(  X_{2}^{\ast},Y_{2}^{\ast}\right)  E_{3}T_{3}\left[
\begin{array}
[c]{c}%
x\left(  t-2b\right) \\
x\left(  t-a-b\right) \\
x\left(  t-2a\right)
\end{array}
\right]  . \label{eq63}%
\end{align}
Thus, by comparing (\ref{eq62}) and (\ref{eq63}) with (\ref{eq60}) and
(\ref{eq61}) respectively, if (\ref{eqpq2}) is satisfied, we obtain
(\ref{eqpq3}). The relation (\ref{eqpq2a}) and (\ref{eqpq3a})\ can be proven
in a similar way.
\end{proof}

\section{Numerical Examples}

\begin{figure}[ptb]
\begin{center}
\includegraphics[scale=0.6]{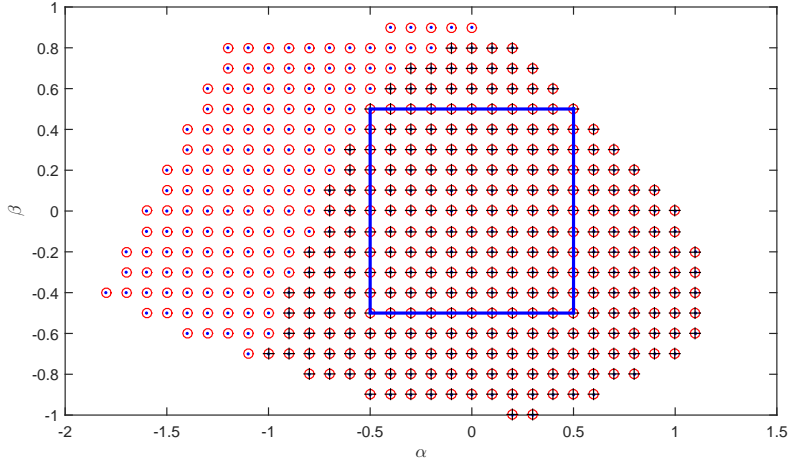}
\end{center}
\caption{Pairs $(\alpha,\beta)$ where the conditions in Lemma \ref{lm3}
(marked by `+'), Lemma \ref{th1} (marked by `.'), and Theorem \ref{th4} with
$k=2$ (which is equivalent to Lemma \ref{th2} with $k=2$, Lemma \ref{lm4}, and
Corollary \ref{coro1}) (marked by `o') are satisfied, respectively. The square
in blue color denotes $\square_{2}.$}%
\label{fig1}%
\end{figure}

We consider the linear difference equation (\ref{sys}) with%
\[
A\left(  \alpha \right)  =\left[
\begin{array}
[c]{cc}%
-0.4 & -0.3\\
0.1+\alpha & 0.15
\end{array}
\right]  ,\;B\left(  \beta \right)  =\left[
\begin{array}
[c]{cc}%
0.1 & 0.25\\
-0.9 & -0.1+\beta
\end{array}
\right]  ,
\]
where $\alpha,\beta \in \mathbf{R}$ are free parameters \cite{rmd18tac}. We look
for the pair $(\alpha,\beta)$ such that system (\ref{sys}) is strongly stable.
By a linear search technique, the regions of $(\alpha,\beta)$ obtained by
different methods are plotted in Fig. \ref{fig1}. One can verify that the
obtained region $(\alpha,\beta)$ by Theorem \ref{th4} with $k=2$ coincides
with the exact region of stability obtained in \cite{rmd18tac}. This indicates
that $k=2$ is already very efficient. Actually, thousands of numerical
examples show that $k=2$ in Theorem \ref{th4} can lead to necessary and
sufficient stability conditions. Thus, the advantage of Theorem \ref{th4} over
Lemma \ref{th1} is that the size of the LMI has been reduced significantly,
especially, for large $n$.

We now treat $\alpha$ and $\beta$ as uncertainties (which might be
time-varying) and solve the robust stability problem, particularly, we want to
find the maximal value of $r>0$ (denoted by $r^{\ast}$) such that the system
(\ref{sys}) is strongly stable for all $\alpha \in \left[  -r,r\right]  $ and
$\beta \in \left[  -r,r\right]  .$ To this end, we rewrite $A\left(
\alpha \right)  =A+\Delta A$ and $B\left(  \beta \right)  =B+\Delta B,$ where%
\begin{align*}
A &  =\left[
\begin{array}
[c]{cc}%
-0.4 & -0.3\\
0.1 & 0.15
\end{array}
\right]  ,\; \Delta A=\left[
\begin{array}
[c]{cc}%
0 & 0\\
\alpha & 0
\end{array}
\right]  ,\\
B &  =\left[
\begin{array}
[c]{cc}%
0.1 & 0.25\\
-0.9 & -0.1
\end{array}
\right]  ,\; \Delta B=\left[
\begin{array}
[c]{cc}%
0 & 0\\
0 & \beta
\end{array}
\right]  .
\end{align*}
It can be verified that $\left(  \Delta B,\Delta A\right)  $ satisfies
(\ref{eq884}) where $F=[\frac{\beta}{r},\frac{\alpha}{r}]$ and%
\[
E_{0}=\left[
\begin{array}
[c]{c}%
0\\
1
\end{array}
\right]  ,\;B_{0}=\left[
\begin{array}
[c]{cc}%
0 & r\\
0 & 0
\end{array}
\right]  ,\;A_{0}=\left[
\begin{array}
[c]{cc}%
0 & 0\\
r & 0
\end{array}
\right]  .
\]
We clearly have $F^{\mathrm{T}}F\leq I_{2}.$ Then, by applying Theorem
\ref{th5} for different $k$ and applying a linear search technique on $r$, we
can get $r_{\ast}\left(  k\right)  .$ It is found that $r_{\ast}\left(
1\right)  =0.4979$ and $r_{\ast}\left(  2\right)  =r_{\ast}\left(  3\right)
=0.5001.$ Denote the square $\square_{k}=\{ \left(  \alpha,\beta \right)
:\alpha \in \left[  -r_{\ast}\left(  k\right)  ,r_{\ast}\left(  k\right)
\right]  ,\beta \in \left[  -r_{\ast}\left(  k\right)  ,r_{\ast}\left(
k\right)  \right]  \}.$ It follows that $\square_{1}$ is very close to
$\square_{2}$ which is recorded in Fig. \ref{fig1}. We can see that the square
$\square_{2}$ turns to be the maximal square that can be included in the
region where the system is strongly stable for fixed $\left(  \alpha
,\beta \right)  .$ This indicates that Theorem \ref{th5} can even provide
necessary and sufficient conditions for robust strong stability for this example.

\section{Conclusion}

This note established a necessary and sufficient condition for guaranteeing
strong stability of linear difference equations with two delays. The most
important advantage of the proposed method is that the coefficients of the
linear difference equation appear as linear functions in the proposed
conditions, which helps to deal the robust stability analysis problem. The
relationships among the proposed condition and the existing ones were revealed
by establishing a time-domain interpretation of the proposed LMI condition.

\section*{Appendix}

\subsection*{A1: A Proof of Lemma \ref{th1}}

Notice that $\rho \left(  \Delta_{\theta}\right)  <1,\forall \theta \in
\mathbf{R}$, is equivalent to that $\Delta_{0}$ is Schur stable and%
\begin{align}
0 &  \neq \left \vert \Delta_{\theta}^{\mathrm{H}}\otimes \Delta_{\theta}%
-I_{n}\otimes I_{n}\right \vert \nonumber \\
&  =\left \vert A^{\mathrm{T}}\otimes B\mathrm{e}^{-\mathrm{j}\theta
}+B^{\mathrm{T}}\otimes A\mathrm{e}^{\mathrm{j}\theta}+\left(  A^{\mathrm{T}%
}\otimes A+B^{\mathrm{T}}\otimes B-I_{n}\otimes I_{n}\right)  \right \vert
\nonumber \\
&  =\mathrm{e}^{-n^{2}\mathrm{j}\theta}\left \vert A^{\mathrm{T}}\otimes
B+B^{\mathrm{T}}\otimes A\mathrm{e}^{-2\mathrm{j}\theta}+\left(
A^{\mathrm{T}}\otimes A+B^{\mathrm{T}}\otimes B-I_{n}\otimes I_{n}\right)
\mathrm{e}^{-\mathrm{j}\theta}\right \vert \nonumber \\
&  =\mathrm{e}^{-n^{2}\mathrm{j}\theta}\left \vert \mathcal{C}_{0}\left(
\mathrm{e}^{\mathrm{j}\theta}I_{2n^{2}}-\mathcal{A}_{0}\right)  ^{-1}%
\mathcal{B}_{0}+\mathcal{D}_{0}\right \vert \nonumber \\
&  =\mathrm{e}^{-n^{2}\mathrm{j}\theta}\left \vert G_{0}\left(  \mathrm{e}%
^{\mathrm{j}\theta}\right)  \right \vert ,\; \forall \theta \in \mathbf{R}%
,\label{eq2}%
\end{align}
where $G_{0}\left(  s\right)  =\mathcal{C}_{0}\left(  sI_{2n^{2}}%
-\mathcal{A}_{0}\right)  ^{-1}\mathcal{B}_{0}+\mathcal{D}_{0}$ with%
\begin{align*}
\mathcal{A}_{0} &  =\left[
\begin{array}
[c]{cc}%
0_{n^{2}\times n^{2}} & I_{n^{2}}\\
0_{n^{2}\times n^{2}} & 0_{n^{2}\times n^{2}}%
\end{array}
\right]  ,\  \mathcal{B}_{0}=\left[
\begin{array}
[c]{c}%
0_{n^{2}\times n^{2}}\\
I_{n^{2}}%
\end{array}
\right]  ,\\
\mathcal{C}_{0} &  =\left[
\begin{array}
[c]{cc}%
B^{\mathrm{T}}\otimes A & A^{\mathrm{T}}\otimes A+B^{\mathrm{T}}\otimes
B-I_{n}\otimes I_{n}%
\end{array}
\right]  ,\\
\mathcal{D}_{0} &  =A^{\mathrm{T}}\otimes B.
\end{align*}
The condition (\ref{eq2}) is also equivalent to
\begin{align*}
0 &  >-G_{0}^{\mathrm{H}}\left(  \mathrm{e}^{\mathrm{j}\theta}\right)
G_{0}\left(  \mathrm{e}^{\mathrm{j}\theta}\right)  \\
&  =\left[
\begin{array}
[c]{c}%
\left(  \mathrm{e}^{\mathrm{j}\theta}I_{2n^{2}}-\mathcal{A}_{0}\right)
^{-1}\mathcal{B}_{0}\\
I_{n^{2}}%
\end{array}
\right]  ^{\mathrm{H}}M_{0}\left[
\begin{array}
[c]{c}%
\left(  \mathrm{e}^{\mathrm{j}\theta}I_{2n^{2}}-\mathcal{A}_{0}\right)
^{-1}\mathcal{B}_{0}\\
I_{n^{2}}%
\end{array}
\right]
\end{align*}
where $\theta \in \mathbf{R}$ and
\[
M_{0}=-\left[
\begin{array}
[c]{cc}%
\mathcal{C}_{0} & \mathcal{D}_{0}%
\end{array}
\right]  ^{\mathrm{T}}\left[
\begin{array}
[c]{cc}%
\mathcal{C}_{0} & \mathcal{D}_{0}%
\end{array}
\right]  .
\]
Thus, by the YKP lemma (Lemma \ref{lm7}), this is equivalent to the existence
of a symmetric matrix $P\in \mathbf{R}^{2n^{2}\times2n^{2}}$ such that%
\begin{equation}
\left[
\begin{array}
[c]{cc}%
\mathcal{A}_{0}^{\mathrm{T}}P\mathcal{A}_{0}-P & \mathcal{A}_{0}^{\mathrm{T}%
}P\mathcal{B}_{0}\\
\mathcal{B}_{0}^{\mathrm{T}}P\mathcal{A}_{0} & \mathcal{B}_{0}^{\mathrm{T}%
}P\mathcal{B}_{0}%
\end{array}
\right]  +M_{0}<0.\label{eq3}%
\end{equation}
Let $P$ be partitioned as%
\begin{equation}
P=\left[
\begin{array}
[c]{cc}%
P_{1} & P_{3}\\
P_{3}^{\mathrm{T}} & -P_{2}%
\end{array}
\right]  ,\label{eq8}%
\end{equation}
where $P_{i},i=1,2,3,$ are $n^{2}\times n^{2}$ matrices with $P_{i},i=1,2$
being symmetric. Then, in view of the special structures of $\left(
\mathcal{A}_{0},\mathcal{B}_{0}\right)  ,$ (\ref{eq3}) is equivalent to the
LMI%
\begin{equation}
\left[
\begin{array}
[c]{ccc}%
-P_{1} & -P_{3} & 0\\
-P_{3}^{\mathrm{T}} & P_{1}+P_{2} & P_{3}\\
0 & P_{3}^{\mathrm{T}} & -P_{2}%
\end{array}
\right]  +M_{0}<0.\label{eq9}%
\end{equation}
Applying the congruent transformation
\[
T=\left[
\begin{array}
[c]{ccc}%
I_{n^{2}} & 0 & 0\\
0 & 0 & I_{n^{2}}\\
0 & I_{n^{2}} & 0
\end{array}
\right]  ,
\]
on the LMI (\ref{eq9}) gives (\ref{eq1}). The proof is finished by noting that
$\Delta_{0}$ is Schur stable if and only if (\ref{eq0}) is satisfied.

\subsection*{A2: Some Technical Notations and Lemmas}

For two matrices $A\in \mathbf{R}^{n\times n},B\in \mathbf{R}^{n\times n},$ the
shuffle product (power) is defined as \cite{fornasini80tcs}%
\[
A^{\left[  i\right]  }B^{\left[  j\right]  }=\sum \limits_{\substack{i_{1}%
+i_{2}+\cdots+i_{s}=i\\j_{1}+j_{2}+\cdots+j_{s}=j}}A^{i_{1}}B^{j_{1}}A^{i_{2}%
}B^{i_{2}}\cdots A^{i_{s}}B^{j_{s}},
\]
where $\left(  i,j\right)  $ is a pair of nonnegative integers, and
$i_{k},j_{k}\geq0,k=1,2,\ldots,s.$ For example,
\[
A^{\left[  1\right]  }B^{\left[  2\right]  }=AB^{2}+BAB+B^{2}A.
\]
There are several simple properties of the shuffle product. For example,
\cite{fornasini80tcs}
\begin{align}
A^{\left[  i\right]  }B^{\left[  j\right]  }  &  =B^{\left[  j\right]
}A^{\left[  i\right]  },\nonumber \\
A^{\left[  i\right]  }B^{\left[  0\right]  }  &  =A^{i},\;A^{\left[  0\right]
}B^{\left[  j\right]  }=B^{j},\nonumber \\
A^{\left[  i\right]  }B^{\left[  j\right]  }  &  =A\left(  A^{\left[
i-1\right]  }B^{\left[  j\right]  }\right)  +B\left(  A^{\left[  i\right]
}B^{\left[  j-1\right]  }\right) \nonumber \\
&  =\left(  A^{\left[  i-1\right]  }B^{\left[  j\right]  }\right)  A+\left(
A^{\left[  i\right]  }B^{\left[  j-1\right]  }\right)  B. \label{eq87}%
\end{align}

We next recall the so-called Yakubovich-Kalman-Popov (YKP) Lemma. This lemma
in the discrete-time setting is also known as the Szego-Kalman-Popov (SKP)
Lemma \cite{kalman63nas,popov64rrstsee,szego63nas}.

\begin{lemma}
\label{lm7} (YKP Lemma) Given $A\in \mathbf{R}^{n\times n},B\in \mathbf{R}%
^{n\times m}$ and $M\in \mathbf{R}^{\left(  n+m\right)  \times \left(
n+m\right)  }$ with $\left \vert \mathrm{e}^{\mathrm{j}\theta}I_{n}%
-A\right \vert \neq0,\forall \theta \in \mathbf{R}.$ Then%
\[
\left[
\begin{array}
[c]{c}%
\left(  \mathrm{e}^{\mathrm{j}\theta}I_{n}-A\right)  ^{-1}B\\
I_{m}%
\end{array}
\right]  ^{\mathrm{H}}M\left[
\begin{array}
[c]{c}%
\left(  \mathrm{e}^{\mathrm{j}\theta}I_{n}-A\right)  ^{-1}B\\
I_{m}%
\end{array}
\right]  <0,
\]
holds for all $\theta \in \mathbf{R}$ if and only if there exists a symmetric
matrix $P\in \mathbf{R}^{n\times n}$ such that%
\[
\left[
\begin{array}
[c]{cc}%
A^{\mathrm{T}}PA-P & A^{\mathrm{T}}PB\\
B^{\mathrm{T}}PA & B^{\mathrm{T}}PB
\end{array}
\right]  +M<0.
\]

\end{lemma}

The next lemma is adopted from \cite{bliman02mssp}.

\begin{lemma}
\label{lm9}\cite{bliman02mssp} If the inequality (\ref{eqr}) is satisfied,
namely,%
\begin{equation}
\sup_{\theta \in \left[  0,2\pi \right]  }\left \{  \rho \left(  \Delta_{\theta
}\right)  \right \}  <1, \label{eqr1}%
\end{equation}
then there exists a $k_{\ast}\in \mathbf{N}^{+}$ such that%
\begin{equation}
\sup_{\theta \in \left[  0,2\pi \right]  }\left \{  \left \Vert \Delta_{\theta}%
^{k}\right \Vert \right \}  <1,\; \forall k\geq k_{\ast}. \label{eqnorm}%
\end{equation}

\end{lemma}

We finally recall a well-known result that was frequently used in robust
control literature.

\begin{lemma}
\label{lm1} \cite{fridman02jmaa} Let $X$ and $Y$ be real matrices of
appropriate dimensions. For $Q>0$ the following inequality is satisfied%
\[
XY+Y^{\mathrm{T}}X^{\mathrm{T}}\leq XQX^{\mathrm{T}}+Y^{\mathrm{T}}Q^{-1}Y.
\]

\end{lemma}

\end{document}